\documentclass[preprint,12pt]{elsarticle}

\usepackage{amssymb}
\usepackage{amsthm}

\usepackage{multirow,multicol}
\usepackage{booktabs}
\usepackage{amsmath, amsthm, amscd, amsfonts, amssymb, graphicx, color}
\usepackage[bookmarksnumbered, plainpages, backref]{hyperref}
\usepackage [latin1]{inputenc}

\usepackage{graphics}
\usepackage{epsfig}
\usepackage{graphicx}
\usepackage{epstopdf}
\usepackage{amsthm}

\theoremstyle{plain}
\newtheorem{theorem}{Theorem}[section]

\newtheorem{remark}[theorem]{Remark}
\newtheorem{definition}[theorem]{Definition}

\usepackage{float}

\journal{Elsevier}

\begin{document}

\begin{frontmatter}

\title{Numerical integration of stochastic contact Hamiltonian systems via stochastic Herglotz variational principle}

\author[mymainaddress]{Qingyi Zhan \corref{mycorrespondingauthor}}
\ead{zhan2017@fafu.edu.cn}
\address[mymainaddress]{College  of Computer and Information Science,\\
                         Fujian Agriculture and Forestry  University,Fuzhou,350002,China}

\author[mysecondaryaddress]{Jinqiao Duan }
\ead{duan@iit.edu}
\address[mysecondaryaddress]{Department of Applied Mathematics,\\
                         Illinois Institute of Technology,Chicago,IL,60616,USA}

\author[mysecondaryaddress]{Xiaofan Li}
\ead{lix@iit.edu}
\cortext[mycorrespondingauthor]{Corresponding author}

\author[mythirdaddress]{Yuhong Li}
\ead{liyuhong@hust.edu.cn}
\address[mythirdaddress]{ College of Hydropower and Information Engineering,\\
   Huazhong University of Science and Technology,
    Wuhan,430074,China}

\begin{abstract}
   In this work we construct a stochastic contact variational integrator and its discrete version via stochastic Herglotz variational principle for stochastic contact Hamiltonian systems. A general structure-preserving stochastic contact method is devised, and the stochastic contact variational integrators are established.
   The implementation of this approach is validated by the numerical experiments.

\end{abstract}


\end{frontmatter}

\section{Introduction}

 Contact geometry was introduced in Sophus Lie's study of differential equations, and has been the subject of intense research, especially related to low-dimensional topology \cite{Arnold, Bravetti, Feng, Geiges}. In the last few years, stochastic contact Hamiltonian systems have been interesting subject \cite{Wei}. These stochastic differential equations (SDEs)  are of the importance in modelling natural phenomena, such as, gravity, thermodynamics and dissipative systems \cite{Duan}. Stochastic contact Hamiltonian systems constitute a rather important class of SDEs, which contain contact structures similar to symplectic structures in the odd dimension. Therefore, there is a demand for the investigation of numerical integrators to preserve the contact structures.

As the counterpart of contact case, stochastic symplectic methods have attracted much attention, such as those in
\cite{Hairer,  Milstein02, Milstein03,Misawa, Hong, Zhan1}, and so on. A theoretical framework, structure-preserving algorithm, has been widely applied in many aspects. The deterministic contact Hamiltonian systems have been studied recently. For example, \cite{Feng} exploits the symplectification  and the corresponding generating functions. However, so far \cite{Feng} has not attracted much attention most likely due to its missing of a variational approach, which would make it hard to construct a numerical contact method. \cite{M. Kraus} investigates variational integrators for stochastic dissipative Hamiltonian systems, which is the Hamiltonian systems in even dimension. That is, they are one class of dissipative Hamiltonian systems.

Fortunately, contact flows possess geometric integrators that precisely parallel their symplectic counterparts.
Contact variational integrators  have been an important approach of creating contact methods \cite{Georgieva}. They have relationship with stochastic Herglotz variational principle and its discrete version.

For stochastic contact Hamiltonian systems, the main difficulty in constructing stochastic contact variational integrators is the formulation of the stochastic Herglotz variational principle \cite{Bravetti02,Leon}.

In this paper we investigate the numerical contact integrators for stochastic contact Hamiltonian systems, by utilizing  the generalized Herglotz variational principle. It is motivated by two factors. First, as we know, an associated variational principle for contact Hamiltonian systems and its contact methods  are of great interest \cite{Georgieva, Liu}. Then it is natural to expect to expand it to the stochastic contact Hamiltonian systems. Second, many contributions are made to the numerical analysis of SDEs \cite{Milstein01, X. Wang, Zhan1}. The readers can find more information on numerical topic in these references. These are the foundations of our present work. However, to the best of our knowledge, systematic construction of contact scheme of stochastic contact Hamiltonian systems  is still an open problem.

Our results furnished a contact scheme for stochastic contact Hamiltonian systems in which the contact structure is preserved almost surely.
 In the numerical experiments, we compare the numerical dynamical behaviors of contact scheme with non-contact scheme in several aspects, such as the orbits in a long time interval, the preservations of contact structure and the conformal factor. For our purpose the numerical experiments are realizable by programming.

The structure of this paper is as follows. Section 2 deals with some preliminaries. In Section 3 the  theoretical results on stochastic Herglotz variational principle are summarized. The discrete version of stochastic Herglotz variational principle is proved. Illustrative numerical experiments are included in Section 4. Finally, the last section is addressed to summarize the conclusions of the paper.

\section{Preliminaries}

We consider the following stochastic Hamiltonian systems in the sense of Stratonovich on a smooth $d=(2n+1)$-dimensional contact manifold $\mathbb{M}$,
\begin{equation}\label{2.1}
 dX(t)=f(t,X(t))dt+\sum_{k=1}^m g_k(t,X(t)) \circ  dW^k(t),\ \ \  \ \ \ X(t_0)=x\in \mathbb{M} ,
\end{equation}
where $X, f(t, x^1,..., x^d), g_k(t, x^1,..., x^d)$ are $d$-dimensional column-vectors with the
components $X^i, f^i, g_k^i, i = 1, . . . , d$, and $W^k(t), k= 1,...,m$, are independent
standard Wiener processes  on a filtered probability space $(\Omega, \mathcal{F}, \{\mathcal{F}_t\}_{t\geq 0}, \mathbb{P})$.

In Darboux coordinates $(q,p,s)=(q^1,q^2,...,q^n,p^1,p^2,...,p^n,s)$, a canonical stochastic contact Hamiltonian system can be rewritten as
 \begin{equation}\label{2.2}
\left \{
\begin{aligned}
 dq&=\frac{\partial H_0}{\partial p}dt+\sum_{k=1}^m \frac{\partial H_k}{\partial p}\circ dW^k(t),\\
 dp&=-\Big(\frac{\partial H_0}{\partial q}+p\frac{\partial H_0}{\partial s}\Big ) dt-\sum_{k=1}^m \Big(\frac{\partial H_k}{\partial q}+p\frac{\partial H_k}{\partial s}\Big)\circ dW^k(t),\\
 ds&=\Big(p\frac{\partial H_0}{\partial p}-H_0 \Big ) dt+\sum_{k=1}^m \Big(p\frac{\partial H_k}{\partial p}-H_k\Big)\circ dW^k(t).
 \end{aligned}
\right.
\end{equation}
with initial condition $(q(t_0),p(t_0),s(t_0))=(q_0, p_0, s_0), t_0 \geq 0$, where $H_0:=H_0(q,p,s)$ is a smooth Hamiltonian function on $\mathbb{M}$, and $\{H_k\}_{k=1}^m:=\{H_k(q,p,s)\}_{k=1}^m$ are a family of smooth functions on $\mathbb{M}$. In fact, Equs. $(\ref{2.2})$ are the generalization of Hamilton's equation to a contact manifold. In particular, if $\{H_k\}_{k=0}^m$ do not depend on $s$, Equs. $(\ref{2.2})$ give a stochastic Hamiltonian equations in the symplectic phase space. Therefore, Equs. $(\ref{2.2})$ generalize the equations of motion for the positions, the momenta and Hamilton's principal function of the standard Hamilton's theory, and can include a large class of models, such as dissipative systems.

We introduce the following notations.

Let $\mathbb{L}^2(\Omega,\mathbb{P})$ be the space of all bounded square-integrable random variables $X:\Omega \rightarrow \mathbb{R}^d$. For random vector $X=({x^1},{ x^2},...,{ x^d}) \in \mathbb{R}^d$, the norm of $X$ is defined in the form of
\begin{equation}\label{2.3}
\|X\|_2=\Big[\int_\Omega[|x^1(\omega)|^2+|x^2(\omega)|^2+...+|x^d(\omega)|^2]d\mathbb{P}\Big]^{\frac{1}{2}}< \infty.
\end{equation}
We define the norm of random matrices as follows\cite{Golub}
\begin{equation}\label{2.4}
 \| G \|_{\mathbb{L}^2(\Omega,\mathbb{P})} =\Big[\mathbb{E}(|G|^2)\Big]^{\frac{1}{2}},
\end{equation}
where $G$ is a random matrix and $|\cdot|$ is the operator norm.

For simplicity, the norms $\|\cdot\|_2$ and $ \| \cdot \|_{\mathbb{L}^2(\Omega,\mathbb{P})}$ are usually written as $\|\cdot\|$.

\section{Theoretical results on stochastic Herglotz variational principle}

\subsection{Stochastic Herglotz variational principle }
Similar to symplectic Hamiltonian systems, variational principle is a powerful tool to study the dynamics of contact Hamiltonian systems. This was originally created by Herglotz in 1930, and there are many related works, such as \cite{Georgieva}.
It follows from \cite{Tveter} that the Lagrange equation can be written in the expanded variable set $(q,p,\dot{q},\dot{p},t)$ \cite{Hong}. We introduce the following definition.

\begin{definition}(Deterministic Herglotz variational principle \cite{Bravetti02,Georgieva,Vermeeren})

Let $\mathbb{Q}$ be an n-dimensional manifold with local coordinates $q^i,p^i$, and $T\mathbb{Q}$ be the tangent space of $\mathbb{Q}$. Consider a continuous Lagrangian $L:\mathbb{R}\times T\mathbb{Q}\times \mathbb{R}  \times T\mathbb{Q}\times \mathbb{R}\rightarrow  \mathbb{R}$. For any given curves $q,p:[0,T]\rightarrow \mathbb{Q} $, which connect $q(0)=q_0,p(0)=p_0$ and $p(T)=p_N, q(T)=q_N$, with $\delta q(0)=\delta q(T)=\delta p(0)=\delta p(T)=0$, satisfying the following deterministic contact Hamiltonian system,
  \begin{equation}
\left \{
\begin{aligned}
 dq&=\frac{\partial H_0}{\partial p}dt,\\
 dp&=-\Big(\frac{\partial H_0}{\partial q}+p\frac{\partial H_0}{\partial s}\Big ) dt,\\
 ds&=\Big(p\frac{\partial H_0}{\partial p}-H_0 \Big ) dt,
 \end{aligned}
\right.
\nonumber
\end{equation}
 the initial value problem is given as follows
\begin{equation}
\dot{s}=L(t,q(t),\dot{q}(t),p(t),\dot{p}(t),s).
\nonumber
\end{equation}
Therefore, the value $s(T)$ is called the action functional of the curves $q(t)$ and $p(t)$. The curves $q(t)$ and $p(t)$ are called critical if and only if $s(T)$ is invariant under infinitesimal variations of $q$ and $p$ that vanish at the boundary of $[0,T]$, where the notation s is the same as in the contact Hamiltonian system, that is, the scalar quantity s in the contact Hamiltonian system can be regarded as the action functional.
\end{definition}

It is natural to generalize it to the stochastic case \cite{Bravetti, Bravetti02, Hong}.

\begin{definition} (Stochastic Herglotz variational principle )

For any given curves $q,p:[0,T]\rightarrow \mathbb{Q} $, with $q(0)=q_0,p(0)=p_0$ and $p(T)=p_N, q(T)=q_N$, satisfying the stochastic contact Hamiltonian system $(\ref{2.2})$, the initial value problem is given as follows
\begin{equation}\label{3.1}
\dot{s}=L(t,q(t),\dot{q}(t),p(t),\dot{p}(t),s)-\sum_{k=1}^mH_k \circ \dot{W}^k,s(0)=s_0,
\end{equation}
where $\dot{W}^kdt=dW^k$. Therefore, the curves $q(t)$ and $p(t)$ are called critical if and only if $s(T)$ is invariant under infinitesimal variations of $q$ and $p$ that vanish at the boundary of $[0,T]$.
\end{definition}

\begin{theorem}
 If the curves $q(t)$ and $p(t)$ are solutions to the following stochastic Euler-Lagrange equations,
 \begin{equation}\label{3.2}
 \left \{
\begin{aligned}
 \frac{\partial L}{\partial q}-\sum_{k=1}^m \frac{\partial H_k}{\partial q}\circ \dot{W}^k-\frac{d}{dt}\frac{\partial L}{\partial \dot{q}}+\Big(\frac{\partial L}{\partial s}-\sum_{k=1}^m\frac{\partial H_k}{\partial s}\circ \dot{W}^k\Big)\frac{\partial L}{\partial \dot{q}}=0,
 \\
 \frac{\partial L}{\partial p}-\sum_{k=1}^m \frac{\partial H_k}{\partial p}\circ\dot{W}^k-\frac{d}{dt}\frac{\partial L}{\partial \dot{p}}+\Big(\frac{\partial L}{\partial s}-\sum_{k=1}^m\frac{\partial H_k}{\partial s}\circ\dot{W}^k\Big)\frac{\partial L}{\partial \dot{p}}=0,
 \end{aligned}
\right.
 \end{equation}
then the action functional $s:[0,T]\longrightarrow \mathbb{R}$ with an initial condition $s(0)=s_0$ can be minimized by this curves $q(t)$ and $p(t)$.
\end{theorem}

\begin{proof}
The variation of $\dot{s}$ is obtained from $(\ref{3.1})$
$$\delta \dot{s}=\frac{\partial L}{\partial q}\delta q+\frac{\partial L}{\partial \dot{q}}\delta
\dot{q}+\frac{\partial L}{\partial p}\delta p+\frac{\partial L}{\partial \dot{p}}\delta
\dot{p}+\frac{\partial L}{\partial s}\delta s-  \sum_{k=1}^m\Big (\frac{\partial H_k}{\partial q}\delta q+\frac{\partial H_k}{\partial p}\delta p+\frac{\partial H_k}{\partial s}\delta s\Big)\circ\dot{W}^k.$$

Here we set
$$A(t):=\frac{\partial L}{\partial q}\delta q+\frac{\partial L}{\partial \dot{q}}\delta
\dot{q}- \sum_{k=1}^m \frac{\partial H_k}{\partial q}\delta q\circ\dot{W}^k,$$
$$B(t):=\frac{\partial L}{\partial p}\delta p+\frac{\partial L}{\partial \dot{p}}\delta
\dot{p}- \sum_{k=1}^m \frac{\partial H_k}{\partial p}\delta p\circ\dot{W}^k,$$
and
$$C(t):=\int_0^t\Big(\frac{\partial L}{\partial s}(\tau)-\sum_{k=1}^m\frac{\partial H_k}{\partial s}(\tau)\circ\dot{W}^k\Big)d\tau.$$
Then we obtain the differential equation as follows

$$\delta \dot{s}(t)=A(t)+B(t)+\frac{dC(t)}{dt}\delta s,$$
whose solution is
$$\delta s(t)=\exp(C(t))\Big[\int_0^t(A(\tau)+B(\tau))\exp(-C(\tau))d\tau+\delta s(0)\Big].$$

Utilize the  integration by parts and the expression of $A(t)$, $B(t)$ and $C(t)$, we obtain that
$$\delta s(T)=\exp(C(t))\cdot\Bigg[\int_0^T\Big(\frac{\partial L}{\partial q}-\sum_{k=1}^m \frac{\partial H_k}{\partial q}\circ\dot{W}^k-\frac{d}{dt}\frac{\partial L}{\partial \dot{q}}+\frac{dC}{dt}\frac{\partial L}{\partial \dot{q}} \Big)\cdot \exp(-C(\tau))\delta qd\tau $$

$$+\frac{\partial L}{\partial \dot{q}}(T)\exp(-C(T))\delta q(T)-\frac{\partial L}{\partial \dot{q}}(0)\delta q(0)$$

$$+\int_0^T\Big(\frac{\partial L}{\partial p}-\sum_{k=1}^m \frac{\partial H_k}{\partial p}\circ\dot{W}^k-\frac{d}{dt}\frac{\partial L}{\partial \dot{p}}+\frac{dC}{dt}\frac{\partial L}{\partial \dot{p}} \Big)\cdot \exp(-C(\tau))\delta pd\tau $$

$$+\frac{\partial L}{\partial \dot{p}}(T)\exp(-C(T))\delta p(T)-\frac{\partial L}{\partial \dot{p}}(0)\delta p(0)+\delta s(0)\Bigg],$$

where $$\frac{dC}{dt}=\frac{\partial L}{\partial s}-\sum_{k=1}^m\frac{\partial H_k}{\partial s}\circ\dot{W}^k.$$

Due to the boundary conditions $\delta q(T)=\delta q(0)=\delta p(T)=\delta p(0)=0$, we obtain that the action $s(T)$ is critical if and only if equation $(\ref{3.2})$ holds.

The proof of Theorem 3.3 is completed.

\end{proof}

\begin{remark} \cite{Cieslinski, Musielak, Vermeeren,Hong}

Here we present the relationship between the stochastic Herglotz variational principle and the Hamiltonian formulation of stochastic contact Hamiltonian systems. Let $L$ be the Lagrangian function with respect to the deterministic part of the stochastic contact systems $(\ref{2.2})$, and it is connected with the deterministic Hamiltonian function $H_0$ through the Legendre transformation
$$L=p^T\dot{q}-H_0.$$
\end{remark}

\begin{theorem}
 The stochastic Lagrangian and Hamiltonian formalisms are equivalent. That is, we can apply the expression of $L$ and stochastic Euler-Lagrangian equations $(\ref{3.2})$ to obtain Equs. $(\ref{2.2})$.
\end{theorem}

\begin{proof}
According to the explicit expression of $L$, we have
$$\frac{\partial L}{\partial q}=-\frac{\partial H_0}{\partial q},\frac{\partial L}{\partial \dot{q}}=p, \frac{\partial L}{\partial s}=-\frac{\partial H_0}{\partial s}.$$
Then substituting these into the first equation of Equs. $(\ref{3.2})$, we have
$$-\frac{\partial H_0}{\partial q}-\sum_{k=1}^m \frac{\partial H_k}{\partial q}\circ\dot{W}^k-\dot{p}+\Big(-\frac{\partial H_0}{\partial s}-\sum_{k=1}^m\frac{\partial H_k}{\partial s}\circ\dot{W}^k\Big)p=0,$$
This is equivalent to the second equation of Equs. $(\ref{2.2})$.

Similarly, we have
$$\frac{\partial L}{\partial p}=\dot{q}-\frac{\partial H_0}{\partial p},\frac{\partial L}{\partial \dot{p}}=0.$$
Then we utilize these to the second equation of Equs. $(\ref{3.2})$ and obtain
$$ \dot{q}-\frac{\partial H_0}{\partial p}-\sum_{k=1}^m \frac{\partial H_k}{\partial p}\circ\dot{W}^k=0,$$
which is equivalent to the first equation of Equs. $(\ref{2.2})$.

Lastly, by Equ. $(\ref{3.1})$, we have
$$\dot{s}=p\dot{q}-H_0-\sum_{k=1}^mH_k\circ\dot{W}^k.$$
Applying  the result of $\dot{q}$, we get
$$\dot{s}=p\Big(\frac{\partial H_0}{\partial p}+\sum_{k=1}^m \frac{\partial H_k}{\partial p}\circ\dot{W}^k\Big )-H_0-\sum_{k=1}^mH_k\circ\dot{W}^k.$$
It is easy to check that this is equivalent to the third equation of Equs. $(\ref{2.2})$.

The proof of Theorem 3.5 is finished.
\end{proof}

\subsection{Discrete stochastic Herglotz variational principle}
 Based on the stochastic Herglotz variational principle and stochastic Euler-Lagrange equations $(\ref{3.2})$,  we will study the construction of the contact methods for the stochastic contact Hamiltonian systems. We first present the following definition, which is the discrete version of stochastic Herglotz variational principle.

 \begin{definition}
 Let $\mathbb{Q}$ be an n-dimensional manifold with local coordinates $q^i,p^i$, and let $L:\mathbb{Q}^4 \times \mathbb{R}\times \mathbb{R}\rightarrow  \mathbb{R}$ with $h>0$. For any given discrete curves $q:=(q_1,q_2,...,q_N)\in \mathbb{Q}^{N+1}$ and $p:=(p_1,p_2,...,p_N)\in \mathbb{Q}^{N+1}$, we define $s:=(s_1,s_2,...,s_N)\in \mathbb{R}^{N+1}$ with $s_0=0$ and
\begin{equation}\label{3.3}
\begin{aligned}
s_{j+1}-s_j=hL_j-\sum_{k=1}^mH_k^j\circ \Delta W_j^k, almost\  surely,
\end{aligned}
\end{equation}
where $$L_j=:L(q_j,q_{j+1},p_j,p_{j+1},s_j,s_{j+1}), H_k^j=:H_k(q_j,q_{j+1},p_j,p_{j+1},s_j,s_{j+1}),$$
and
$$\Delta W_j^k=W^k(t_{j+1})-W^k(t_j).$$
Then the curve $s_N$ is a functional of the discrete curves $q$ and $p$. The discrete curves $q$ and $p$ are the solutions to the discrete
stochastic Herglotz variational principle if and only if
\begin{equation}\label{3.4}
\frac{\partial s_N}{\partial q_j}=0, \forall j\in \{1,2,...,N-1\}, almost \ surely.
\end{equation}
\end{definition}

 \begin{definition}
 If the numerical solution $\{(q_j,p_j,s_j)\}_{j=0}^N$ of the stochastic contact Hamiltonian system $(\ref{2.2})$ is obtained by a discrete scheme and satisfies the following equality,
\begin{equation}\
ds_{j+1}-p_{j+1}dq_{j+1}=\lambda_j(ds_j-p_j dp_j),j=0,1,2,...,N-1, almost \ surely,
\nonumber
\end{equation}
this discrete scheme is called contact scheme,
where $\lambda_j$ is the conformal factor with $\lambda_0=1$.
\end{definition}

 Now we will show the fact that the discrete scheme obtained by Theorem 3.8 is contact one, that is, the scheme preserves contact structure in the case of discrete time.

\begin{theorem}
Let
\begin{equation}\label{3.5}
 p_j^-=\frac{D^{j-1}}{1-E^{j-1}},p_j^+=-\frac{D^{j}}{1+E^{j}},
\end{equation}
where
 \begin{equation}
 \left \{
\begin{aligned}
 D^{j-1}=&h\frac{\partial L_{j-1}}{\partial q_j}-\sum_{k=1}^m\frac{\partial H_k^{j-1}}{\partial q_j}\circ\Delta W_{j-1}^k,\\
E^{j-1}=&h\frac{\partial L_{j-1}}{\partial s_j}-\sum_{k=1}^m\frac{\partial H_k^{j-1}}{\partial s_j}\circ\Delta W_{j-1}^k,
 \end{aligned}
\right.
\nonumber
 \end{equation}
and
\begin{equation}
 \left \{
\begin{aligned}
 D^{j}=&h\frac{\partial L_j}{\partial q_j}-\sum_{k=1}^m\frac{\partial H_k^j}{\partial q_j} \circ\Delta W_j^k,\\
E^j=&h\frac{\partial L_j}{\partial s_j}-\sum_{k=1}^m\frac{\partial H_k^j}{\partial s_j}\circ \Delta W_j^k.
 \end{aligned}
\right.
\nonumber
 \end{equation}

Then the solutions to the discrete generalized Herglotz variational principle are obtained by
\begin{equation}\label{3.6}
p_j=p_j^-=p_j^+.
 \end{equation}

And the map $(q_j,p_j,s_j)\rightarrow (q_{j+1},p_{j+1},s_{j+1})$ induced by a critical discrete curve preserves the contact structure, i.e., $$ds_{j+1}-p_{j+1}dq_{j+1}=\lambda_j(ds_j-p_jdq_j),j=0,1,...,N-1, almost \ surely,$$
where the conformal factor $\lambda_j$ is
\begin{equation} \label{3.7}
\lambda_j=\frac{1+h\frac{\partial L_j}{\partial s_j}-\sum_{k=1}^m\frac{\partial H_k^j}{\partial s_j}\circ\Delta W_j^k}{1-h\frac{\partial L_j}{\partial s_{j+1}}+\sum_{k=1}^m \frac{\partial H_k^j}{\partial s_{j+1}}\circ\Delta W_j^k}.
\end{equation}

\end{theorem}

\begin{proof}
First, we will prove the discrete stochastic Herglotz variational principle.
It follows from equation $(\ref{3.3})$ that
\begin{equation}
\begin{aligned}
\frac{\partial s_{j+1}}{\partial q_j}&=\frac{\partial s_j}{\partial q_j}+h\frac{\partial L_j}{\partial q_j}
+h\frac{\partial L_j}{\partial s_j}\frac{\partial s_j}{\partial q_j}+h\frac{\partial L_j}{\partial s_{j+1}}\frac{\partial s_{j+1}}{\partial q_j}\\
&-\sum_{k=1}^m\Big(\frac{\partial H_k^j}{\partial q_j}+\frac{\partial H_k^j}{\partial s_j} \frac{\partial s_j}{\partial q_j}+\frac{\partial H_k^j}{\partial s_{j+1}} \frac{\partial s_{j+1}}{\partial q_j}\Big )\circ\Delta W_j^k.
\end{aligned}
\nonumber
\end{equation}

Then we have
\begin{equation}\label{3.8}
\begin{aligned}
\Big[1-h\frac{\partial L_j}{\partial s_{j+1}}-\sum_{k=1}^m\frac{\partial H_k^j}{\partial s_{j+1}}\circ\Delta W_j^k\Big]&\frac{\partial s_{j+1}}{\partial q_j}
=h\frac{\partial L_j}{\partial q_j}-\sum_{k=1}^m\frac{\partial H_k^j}{\partial q_j}\circ\Delta W_j^k\\
&+\Big[1+h\frac{\partial L_j}{\partial s_j}-\sum_{k=1}^m\frac{\partial H_k^j}{\partial s_j}\circ\Delta W_j^k\Big ]\frac{\partial s_j}{\partial q_j}.
\end{aligned}
\end{equation}
We also follows from $(\ref{3.3})$ that
\begin{equation}\label{3.9}
\begin{aligned}
s_{j}-s_{j-1}=hL_{j-1}-\sum_{k=1}^mH_k^{j-1}\circ\Delta W_{j-1}^k,
\end{aligned}
\end{equation}
where $$L_{j-1}=:L(q_{j-1},q_j,p_{j-1},p_j,s_{j-1},s_j), H_k^{j-1}=:H_k(q_{j-1},q_j,p_{j-1},p_j,s_{j-1},s_j),$$
and
$$\Delta W_{j-1}^k=W^k(t_j)-W^k(t_{j-1}).$$
So we can obtain from $(\ref{3.9})$ that
\begin{equation}\label{3.10}
\begin{aligned}
\frac{\partial s_{j}}{\partial q_j}=h\frac{\partial L_{j-1}}{\partial q_j}
+h\frac{\partial L_{j-1}}{\partial s_j}\frac{\partial s_j}{\partial q_j}
-\sum_{k=1}^m\Big(\frac{\partial H_k^{j-1}}{\partial q_j}+\frac{\partial H_k^{j-1}}{\partial s_j} \frac{\partial s_j}{\partial q_j}\Big )\circ\Delta W_{j-1}^k.
\end{aligned}
\end{equation}
Then it follows from $(\ref{3.10})$ that
\begin{equation}\label{3.11}
\begin{aligned}
\frac{\partial s_j}{\partial q_j}
=\frac{h\frac{\partial L_{j-1}}{\partial q_j}-\sum_{k=1}^m\frac{\partial H_k^{j-1}}{\partial q_j}\circ\Delta W_{j-1}^k}{1-h\frac{\partial L_{j-1}}{\partial s_j}+\sum_{k=1}^m\frac{\partial H_k^{j-1}}{\partial s_j}\circ\Delta W_{j-1}^k}.
\end{aligned}
\end{equation}
Substituting $(\ref{3.11})$ into $(\ref{3.8})$ gives
\begin{equation}\label{3.12}
\begin{aligned}
\Big[1-h\frac{\partial L_j}{\partial s_{j+1}}-\sum_{k=1}^m\frac{\partial H_k^j}{\partial s_{j+1}}\circ\Delta W_j^k\Big]\frac{\partial s_{j+1}}{\partial q_j}
=h\frac{\partial L_j}{\partial q_j}-\sum_{k=1}^m\frac{\partial H_k^j}{\partial q_j}\circ\Delta W_j^k\\
+\Big[1+h\frac{\partial L_j}{\partial s_j}-\sum_{k=1}^m\frac{\partial H_k^j}{\partial s_j}\circ\Delta W_j^k\Big ]\cdot
\frac{h\frac{\partial L_{j-1}}{\partial q_j}-\sum_{k=1}^m\frac{\partial H_k^{j-1}}{\partial q_j}\circ\Delta W_{j-1}^k}{1-h\frac{\partial L_{j-1}}{\partial s_j}+\sum_{k=1}^m\frac{\partial H_k^{j-1}}{\partial s_j}\circ\Delta W_{j-1}^k}.
\end{aligned}
\nonumber
\end{equation}
Due to $(\ref{3.5})$ and $(\ref{3.6})$, we obtain that
solutions to the discrete generalized Herglotz variational principle are obtained by
$$p_j=p_j^-=p_j^+,$$
if and only if the equality holds $$\frac{\partial s_{j+1}}{\partial q_j}=0.$$

Second, we will prove the preserve of the contact structure. From $(\ref{3.3})$ it follows
\begin{equation}\label{3.13}
\begin{aligned}
ds_{j+1}&-ds_j=h\frac{\partial L_j}{\partial q_j}dq_j+h\frac{\partial L_j}{\partial q_{j+1}}dq_{j+1}
+h\frac{\partial L_j}{\partial s_j}ds_j+h\frac{\partial L_j}{\partial s_{j+1}}ds_{j+1}\\
&-\sum_{k=1}^m\Big[\frac{\partial H_k^j}{\partial q_j}dq_j+\frac{\partial H_k^j}{\partial q_{j+1}}dq_{j+1}+\frac{\partial H_k^j}{\partial s_j}ds_j+\frac{\partial H_k^j}{\partial s_{j+1}}ds_{j+1}\Big]\circ\Delta W_j^k.
\end{aligned}
\end{equation}
Motivated by $(\ref{3.5})$, $(\ref{3.6})$ and $(\ref{3.13})$, we have
\begin{equation}
\begin{aligned}
&\Big[1-h\frac{\partial L_j}{\partial s_{j+1}}+\sum_{k=1}^m \frac{\partial H_k^j}{\partial s_{j+1}}\circ\Delta W_j^k\Big]ds_{j+1}-\Big[h\frac{\partial L_j}{\partial q_{j+1}}-\sum_{k=1}^m\frac{\partial H_k^j}{\partial q_{j+1}}\circ\Delta W_j^k\Big]dq_{j+1}\\
&=\Big[1+h\frac{\partial L_j}{\partial s_j}-\sum_{k=1}^m\frac{\partial H_k^j}{\partial s_j}\circ\Delta W_j^k\Big]ds_j
+\Big[h\frac{\partial L_j}{\partial q_j}-\sum_{k=1}^m\frac{\partial H_k^j}{\partial q_j}\circ\Delta W_j^k\Big]dq_j\\
&=\Big[1+h\frac{\partial L_j}{\partial s_j}-\sum_{k=1}^m\frac{\partial H_k^j}{\partial s_j}\circ\Delta W_j^k\Big]\cdot \Big [ds_j-
\frac{h\frac{\partial L_{j-1}}{\partial q_j}-\sum_{k=1}^m\frac{\partial H_k^{j-1}}{\partial q_j}\circ\Delta W_{j-1}^k}{1-h\frac{\partial L_{j-1}}{\partial s_j}+\sum_{k=1}^m\frac{\partial H_k^{j-1}}{\partial s_j}\circ\Delta W_{j-1}^k}dq_j\Big].
\end{aligned}
\nonumber
\end{equation}
Therefore, we obtain that
\begin{equation}\label{3.14}
\begin{aligned}
&ds_{j+1}-\frac{h\frac{\partial L_j}{\partial q_{j+1}}-\sum_{k=1}^m\frac{\partial H_k^j}{\partial q_{j+1}}\circ\Delta W_j^k}{1-h\frac{\partial L_j}{\partial s_{j+1}}+\sum_{k=1}^m \frac{\partial H_k^j}{\partial s_{j+1}}\circ\Delta W_j^k}dq_{j+1}\\
&=\frac{1+h\frac{\partial L_j}{\partial s_j}-\sum_{k=1}^m\frac{\partial H_k^j}{\partial s_j}\circ\Delta W_j^k}{1-h\frac{\partial L_j}{\partial s_{j+1}}+\sum_{k=1}^m \frac{\partial H_k^j}{\partial s_{j+1}}\circ\Delta W_j^k}\cdot \Bigg [ds_j-
\frac{h\frac{\partial L_{j-1}}{\partial q_j}-\sum_{k=1}^m\frac{\partial H_k^{j-1}}{\partial q_j}\circ\Delta W_{j-1}^k}{1-h\frac{\partial L_{j-1}}{\partial s_j}+\sum_{k=1}^m\frac{\partial H_k^{j-1}}{\partial s_j}\circ\Delta W_{j-1}^k}dq_j\Bigg].
\end{aligned}
\end{equation}
Together with $(\ref{3.5})-(\ref{3.6})$, equation $(\ref{3.14})$ implies
$$ds_{j+1}-p_{j+1}dq_{j+1}=\lambda_j( ds_j-p_jdq_j).$$
Here the conformal factor $\lambda_j$ is
\begin{equation}
\lambda_j=\frac{1+h\frac{\partial L_j}{\partial s_j}-\sum_{k=1}^m\frac{\partial H_k^j}{\partial s_j}\circ\Delta W_j^k}{1-h\frac{\partial L_j}{\partial s_{j+1}}+\sum_{k=1}^m \frac{\partial H_k^j}{\partial s_{j+1}}\circ\Delta W_j^k},
\nonumber
\end{equation}
which is the same as $(\ref{3.7})$.

The proof of Theorem 3.8 is finished.
\end{proof}

\begin{remark}
By Taylor expansion, it is clear that
\begin{equation}
\begin{aligned}
\lambda_j=1+\Big[(h\frac{\partial L_j}{\partial s_j}-\sum_{k=1}^m\frac{\partial H_k^j}{\partial s_j}\circ\Delta W_j^k)+(h\frac{\partial L_j}{\partial s_{j+1}}-\sum_{k=1}^m\frac{\partial H_k^j}{\partial s_{j+1}}\circ\Delta W_j^k)\Big]+o(h^2)\\
=\exp\Big[h(\frac{\partial L_j}{\partial s_j}+\frac{\partial L_j}{\partial s_{j+1}})-\sum_{k=1}^m(\frac{\partial H_k^j}{\partial s_j}+\frac{\partial H_k^j}{\partial s_{j+1}})\circ\Delta W_j^k\Big] +o(h^2).
\end{aligned}
\nonumber
\end{equation}
Therefore, the conformal factor $\lambda_j$  is consistent with the continuous case \cite{Wei}
\begin{equation}
\lambda(t_j)=\exp\Big[\int_{t_j}^{t_{j+1}} \frac{\partial L}{\partial s} d\tau-\sum_{k=1}^m\int_{t_j}^{t_{j+1}}\frac{\partial H_k}{\partial s}\circ dW_j^k\Big].
\nonumber
\end{equation}
\end{remark}

\newpage

\section{Numerical Experiments}

We apply Theorem 3.8 to
construct contact schemes via stochastic Herglotz variational principle for several stochastic contact Hamiltonian systems. The resulted schemes show the validity of stochastic Herglotz variational principle and the effectiveness of the numerical experiments.

\subsection{Example 1}

We consider the following stochastic damped mechanical system driven by additive Gaussian noise,
\begin{equation}\label{4.1}
 \left \{
\begin{aligned}
dq=&pdt,\\
dp=&(-V'(q)-\alpha p)dt,\\
ds=&(p^2-H_0)dt-\varepsilon \circ dW(t),
 \end{aligned}
\right.
 \end{equation}
where the coordinates $q,p$ and $s$ are one dimension. It is possible to expand them to high dimensions. The constant $\alpha$ is positive, $V(q)$ is a potential function, and the Hamiltonians are of the form
$$H_0=\frac{1}{2}p^2+V(q)+\alpha s,\ \ \ \ H_1=\varepsilon.$$
Obviously, it is a special linear stochastic contact Hamiltonian systems with additive Gaussian noise.

Motivated by \cite{Cieslinski, Musielak, Vermeeren}, we can obtain the Lagrangian function $L$ with respect to the deterministic part of $(\ref{4.1})$ is
$$L=p\dot{q}-H_0=\frac{1}{2}\dot{q}^2-V(q)-\alpha s,$$
where $\dot{q}=\frac{dq}{dt}=p$.

One of the discretizations of the Lagrangian $L$ is
\begin{equation}\label{4.2}
\begin{split}
L_j&=L(q_j,q_{j+1},p_j,p_{j+1},s_j,s_{j+1})\\
&=\frac{1}{2}(\frac{q_{j+1}-q_j}{h})^2-\frac{V(q_j)+V(q_{j+1})}{2}-\alpha s_j,
 \end{split}
\end{equation}
where the midpoint quadrature is applied to approximate the differential, $\dot{q}\approx \frac{q_{j+1}-q_j}{h}$, with $\Delta W_j=W(t_{j+1})-W(t_j)$, and
$\Delta t_j=t_{j+1}-t_j, j=0,1,..,N$.

Substituting the expression of $(\ref{4.2})$ into the relations $(\ref{3.5})$ and $(\ref{3.6})$, we get
\begin{equation}
\begin{aligned}
p_j=
\frac{h\frac{\partial L_{j-1}}{\partial q_j}-\frac{\partial H_1^{j-1}}{\partial q_j}\circ\Delta W_{j-1}}{1-h\frac{\partial L_{j-1}}{\partial s_j}+\frac{\partial H_1^{j-1}}{\partial s_j}\circ\Delta W_{j-1}}=\frac{q_j-q_{j-1}}{h}-\frac{V'(q_j)}{2}h.
\end{aligned}
\nonumber
\end{equation}
and
\begin{equation}
\begin{aligned}
p_{j-1}=-
\frac{h\frac{\partial L_{j-1}}{\partial q_{j-1}}-\frac{\partial H_1^{j-1}}{\partial q_{j-1}}\circ\Delta W_{j-1}}{1+h\frac{\partial L_{j-1}}{\partial s_{j-1}}-\frac{\partial H_1^{j-1}}{\partial s_{j-1}}\circ\Delta W_{j-1}}=\frac{\frac{q_j-q_{j-1}}{h}+\frac{V'(q_{j-1})}{2}h}{1-h\alpha}.
\end{aligned}
\nonumber
\end{equation}

Therefore, the contact scheme of $(\ref{4.1})$ can be obtained explicitly,
\begin{equation}\label{4.3}
\left \{
\begin{aligned}
q_j&=q_{j-1}+h(1-h\alpha)p_{j-1}-\frac{h^2}{2}V'(q_{j-1}),\\
p_j&=(1-h\alpha)p_{j-1}-\frac{h}{2}\Big[V'(q_{j-1})+V'(q_j)\Big],\\
s_j&=s_{j-1}+\frac{1}{2h}(q_j-q_{j-1})^2-\frac{h}{2}\Big[V(q_j)+V(q_{j-1})\Big]\\
&-\alpha h s_{j-1}-\varepsilon \Delta W_{j-1}.
\end{aligned}
\right.
\end{equation}

The next part is devoted to the preservation of contact structure  of the scheme $(\ref{4.3})$, whose conformal factor is $\lambda_j=\exp(-\alpha \Delta t_j)$ in the continuous time case, and the comparison with Euler-Maruyama scheme, i.e., non-contact scheme. The Euler-Maruyama scheme of $(\ref{4.1})$ is
\begin{equation}\label{4.4}
\left \{
\begin{aligned}
q_j&=q_{j-1}+hp_{j-1},\\
p_j&=p_{j-1}-h\Big[V'(q_{j-1})+\alpha p_{j-1}\Big],\\
s_j&=s_{j-1}+\frac{1}{2}h(p_{j-1})^2-hV(q_{j-1})-\alpha h s_{j-1}-\varepsilon \Delta W_{j-1}.
\end{aligned}
\right.
\end{equation}

To study the difference of the contact and non-contact scheme, we can compare long time behaviors of the numerical solutions obtained by the schemes $(\ref{4.3})$ and $(\ref{4.4})$. In order to improve the accuracy of the comparison, let $V(q)=\frac{1}{2}q^2$, and the initial conditions are listed as follows. That is, the step size is $h=0.1$, $T=20.0$, $\alpha=0.1$, $\varepsilon=0.02$, $N=200.0$ and the initial value is $q(0)=0.75,p(0)=-0.25,s(0)=0.08$.

 \begin{figure}[H]
   \centering
    \begin{minipage}{6.5cm}
       \includegraphics[width=2.8in, height=2.0in]{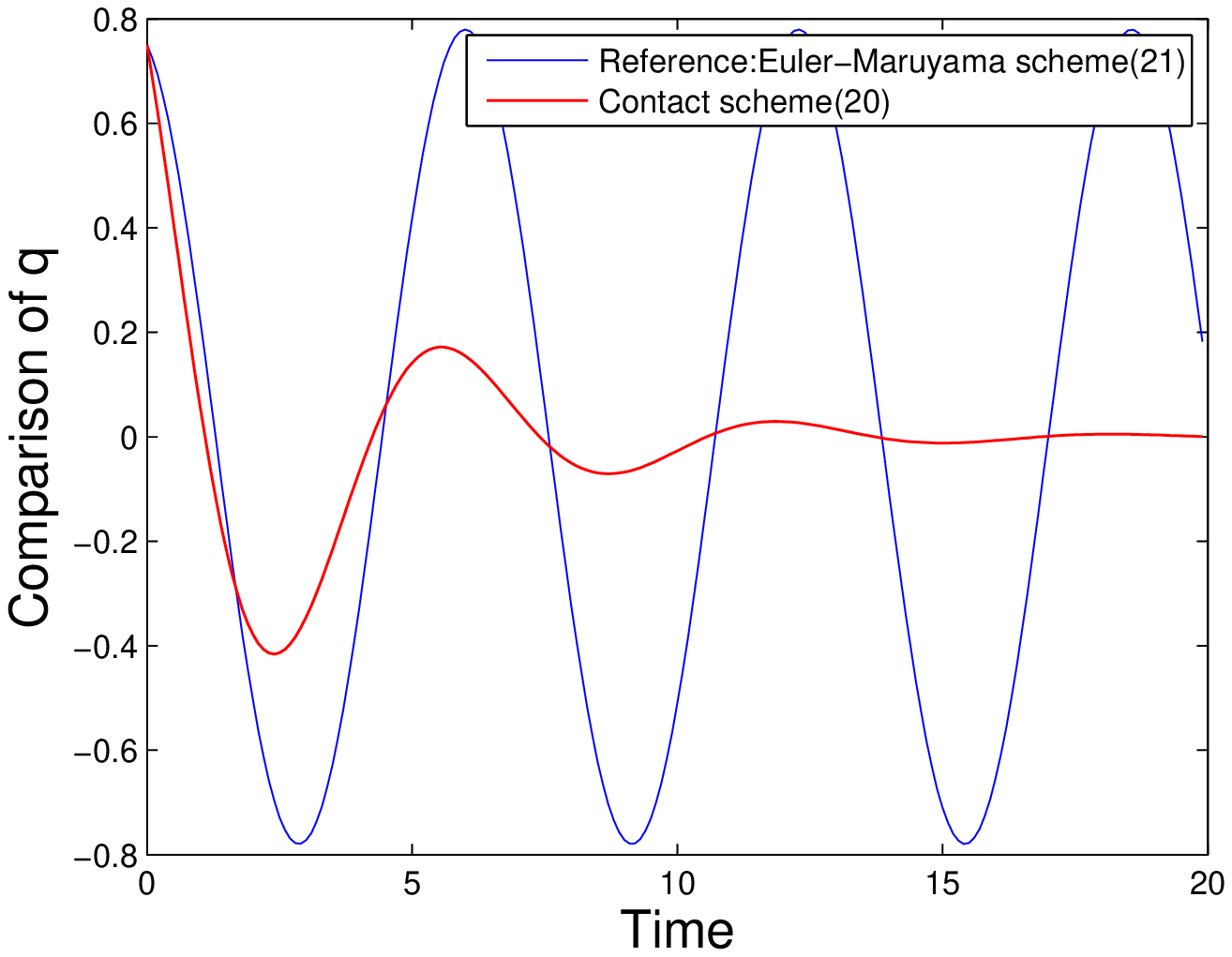}
    \end{minipage}
   \\
   \begin{minipage}{6.5cm}
       \includegraphics[width=2.8in, height=2.0in]{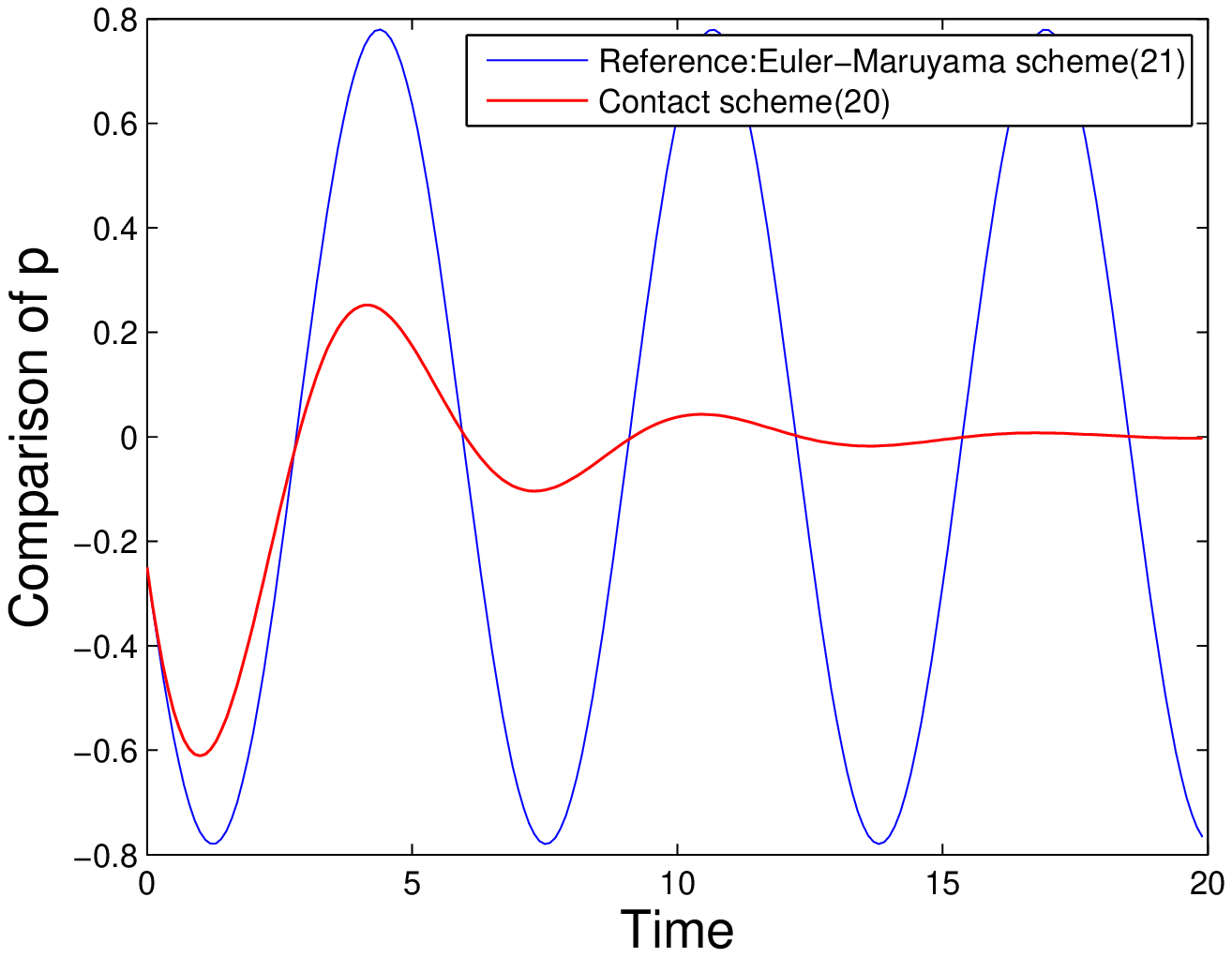}
    \end{minipage}
    \begin{minipage}{6.5cm}
       \includegraphics[width=2.8in, height=2.0in]{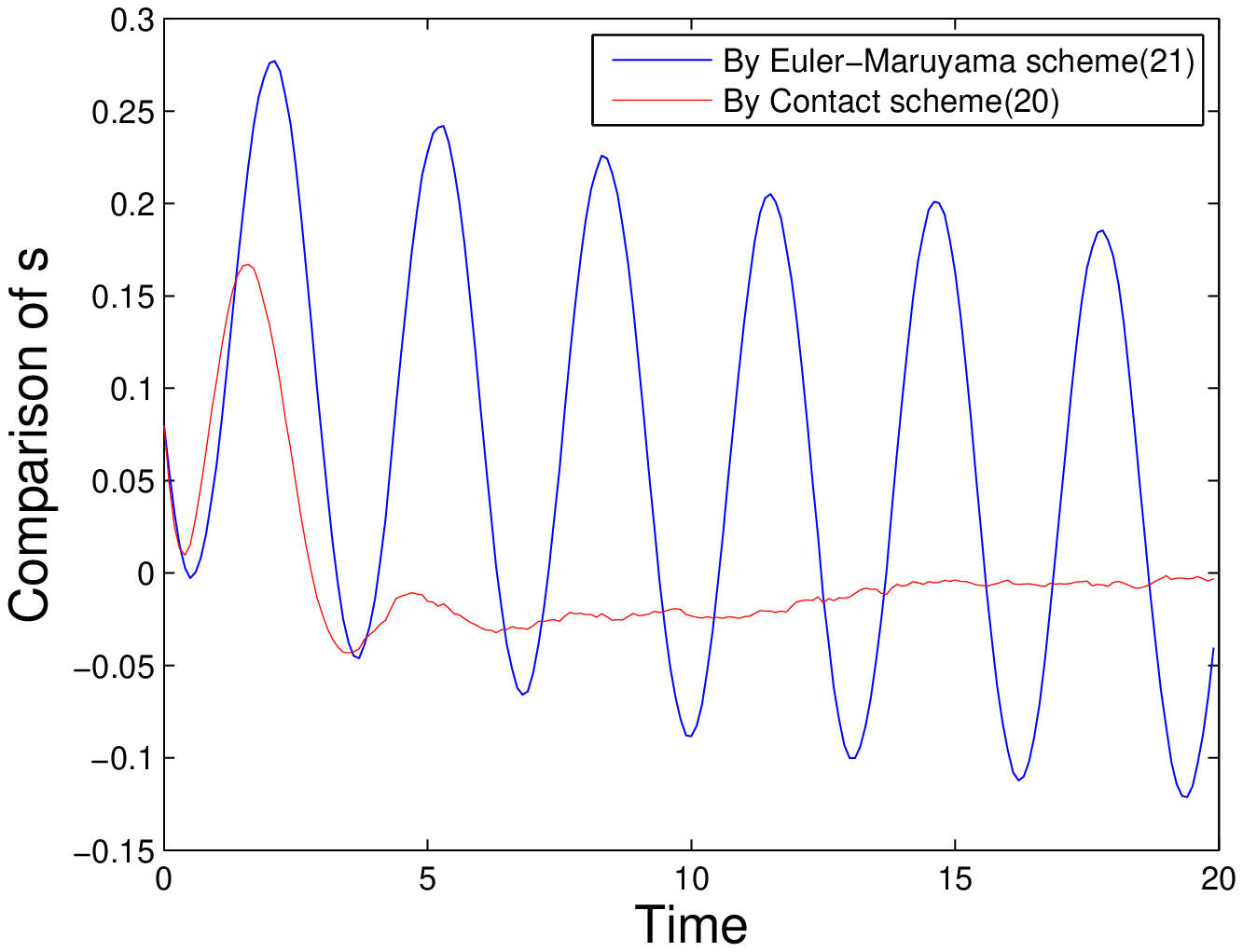}
    \end{minipage}
    \par{\scriptsize  Fig.1. Comparison of sample trajectories of $(\ref{4.3})$  and $(\ref{4.4})$ in the coordinates $q,p$ and $s$, respectively.}
\end{figure}
 As shown in Fig.1, the contact scheme has better performance than the non-contact scheme. That is, the dissipation  phenomenon of SDE $(\ref{4.1})$ is better simulated by contact scheme $(\ref{4.3})$ than non-contact scheme $(\ref{4.4})$.

\begin{figure}[H]
\centering
\includegraphics[width=3.8in, height=1.6in]{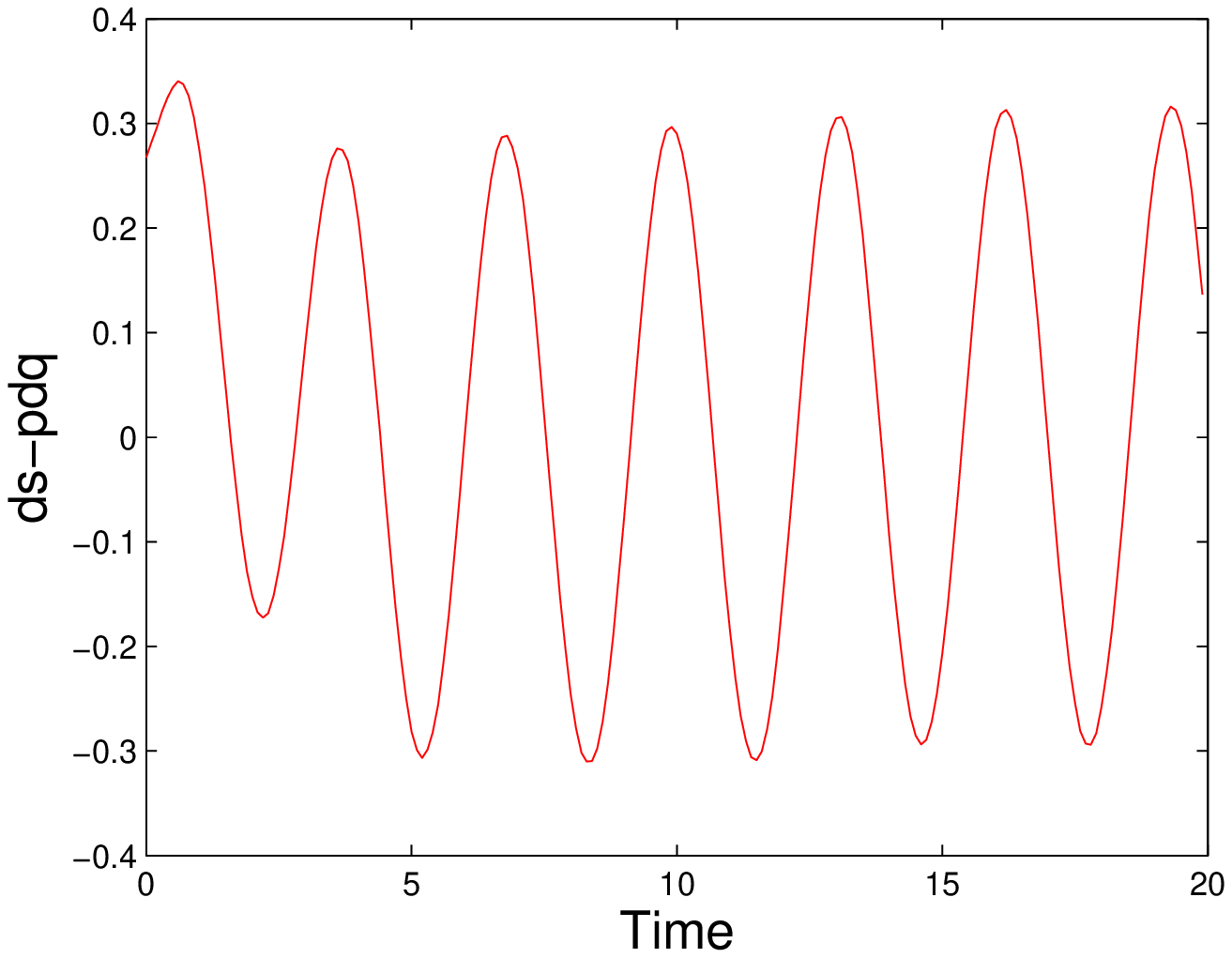}
\par{\scriptsize   Fig.2.  Preservation of contact structure $(\ref{4.1})$ .}
\end{figure}

 \begin{figure}[H]
\centering
\includegraphics[width=3.8in, height=2.0in]{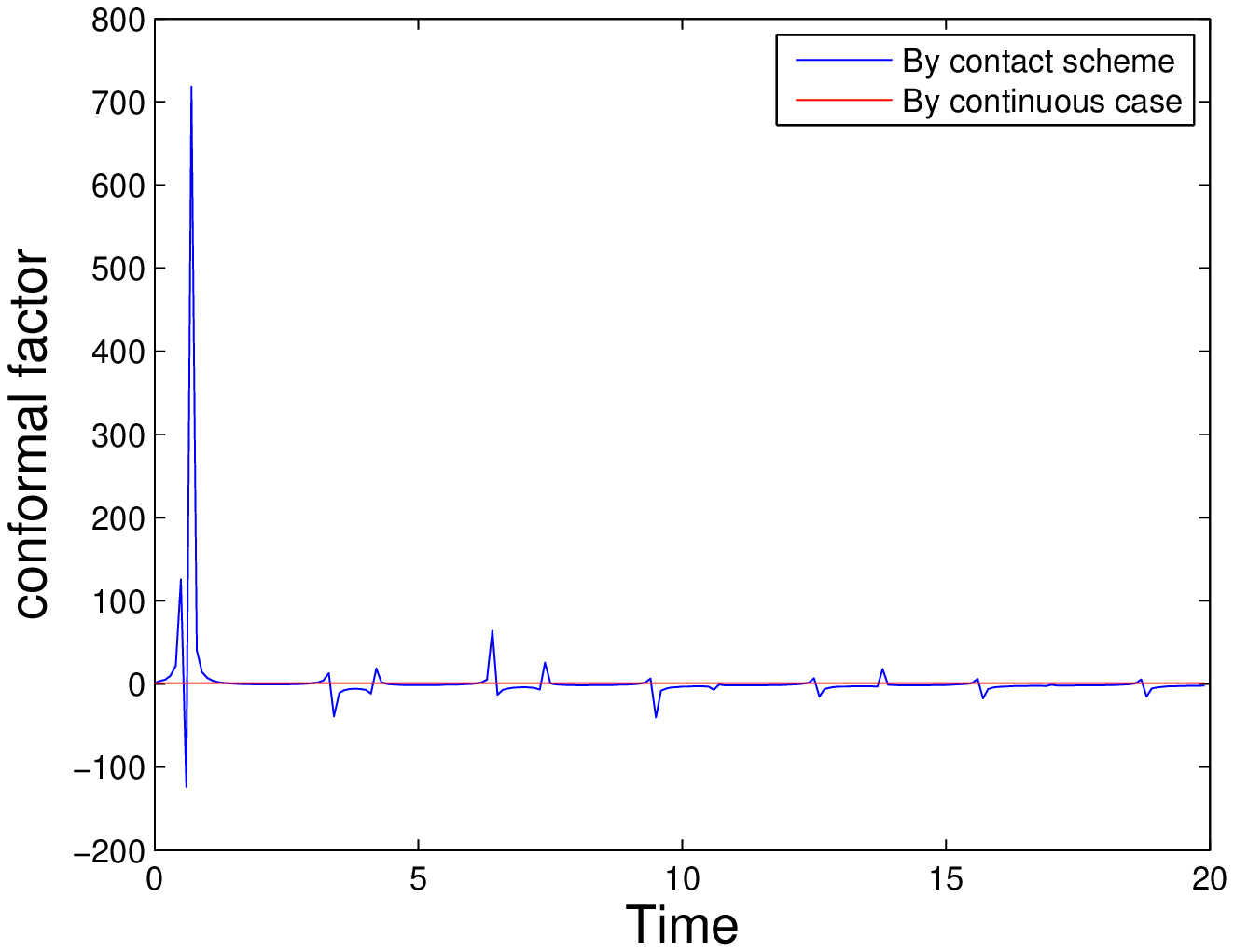}
\par{\scriptsize   Fig.3.  Comparison of the conformal factor of contact scheme $(\ref{4.3})$  and the continuous case of $(\ref{4.1})$.}
\end{figure}

 It follows from Fig.2 that the contact structure of $(\ref{4.1})$ is almost surely preserved, although there are some small size perturbations.
 Fig.3 shows the fact that the conformal factor is preserved well by the contact scheme. Due to continuous inputting of the Gaussian noise, the curve has some vibrations in some uncertain time moments, but it almost lies near the straight line of  the continuous case, whose value is $\lambda=\exp(-\alpha h)$. Therefore, these phenomena verify the results of Theorem 3.8.

\subsection{Example 2}
Now we illustrate through another example the method of stochastic Herglotz variational principle in constructing contact schemes. For a more general case, we consider the following stochastic damped mechanical systems driven by multiplicative Gaussian noise,
\begin{equation}\label{4.5}
 \left \{
\begin{aligned}
dq=&pdt,\\
dp=&-\Big[V'(q)+\alpha s\dot{q}\Big ]dt-\cos q\circ dW(t),\\
ds=&\Big[\frac{1}{2}(\dot{q})^2-V(q)-\frac{1}{2}\alpha s^2\Big ]dt-\sin q \circ dW(t),
 \end{aligned}
\right.
 \end{equation}
where the coordinates $q,p$ and $s$ are one dimension. And the parameters  $\alpha$  and $V(q)$ are as same as Section 4.1. The Hamiltonians are of the form
 $$H_0=\frac{1}{2}p^2+V(q)+\frac{1}{2}\alpha s^2,\ \ \ H_1=\sin q.$$
It follows from \cite{Cieslinski, Musielak, Vermeeren} that the Lagrangian function $L$ with respect to the deterministic part of $(\ref{4.5})$ is
$$L=p\dot{q}-H_0=\frac{1}{2}\dot{q}^2-V(q)-\frac{1}{2}\alpha s^2.$$

We consider a discrete form of the Lagrangian as follows
\begin{equation}\label{4.6}
\begin{split}
L_j=\frac{1}{2}(\frac{q_{j+1}-q_j}{h})^2-\frac{V(q_j)+V(q_{j+1})}{2}-\frac{1}{4}\alpha s_j^2-\frac{1}{4}\alpha s_{j+1}^2.
 \end{split}
\end{equation}
Substituting the expression of $(\ref{4.6})$ into the relations $(\ref{3.5})$ and $(\ref{3.6})$, we have
\begin{equation}
\begin{aligned}
p_j=\frac{\frac{q_j-q_{j-1}}{h}-\frac{V'(q_j)}{2}h}{1+\frac{h}{2}\alpha s_j}.
\end{aligned}
\nonumber
\end{equation}
and
\begin{equation}
\begin{aligned}
p_{j-1}=\frac{\frac{q_j-q_{j-1}}{h}+\frac{V'(q_{j-1})}{2}h+\cos(q_{j-1})\Delta W_{j-1}}{1-\frac{h}{2}\alpha s_{j-1}}.
\end{aligned}
\nonumber
\end{equation}

Therefore, an implicit contact scheme of $(\ref{4.5})$ can be obtained,
\begin{equation}\label{4.7}
\left \{
\begin{aligned}
q_j=&q_{j-1}+h(1-\frac{1}{2}h\alpha s_{j-1})p_{j-1}-h\cos(q_{j-1})\circ\Delta W_{j-1}
-\frac{V'(q_{j-1})}{2}h^2,\\
p_j=&\Big[(1-\frac{1}{2}h\alpha s_{j-1})p_{j-1}-\cos(q_{j-1})\circ\Delta W_{j-1}\\
    &-\frac{1}{2}h(V'(q_j)+V'(q_{j-1}))\Big]/(1+\frac{1}{2}h\alpha s_j),\\
s_j=&s_{j-1}+\frac{1}{2h}(q_j-q_{j-1})^2-\frac{h}{2}\Big[V(q_j)+V(q_{j-1})\Big]\\
   &-\frac{1}{4}\alpha h s^2_{j-1}-\frac{1}{4}\alpha h s^2_{j}-\sin(q_{j-1})\circ \Delta W_{j-1}.
\end{aligned}
\right.
\end{equation}

It is the position that we devote to the preservation of contact structure of the scheme $(\ref{4.7})$, whose conformal factor is $\lambda_j=\exp(-\alpha \Delta t_j)$ in the continuous case, and the comparison with Euler-Maruyama scheme. The Euler-Maruyama scheme of $(\ref{4.5})$ is shown in the form of

\begin{equation}\label{4.8}
\left \{
\begin{aligned}
q_j&=q_{j-1}+hp_{j-1},\\
p_j&=p_{j-1}-h\Big[V'(q_{j-1})+\alpha s_{j-1}\frac{q_j-q_{j-1}}{h}\Big]-\cos (q_{j-1}) \circ \Delta W_{j-1},\\
s_j&=s_{j-1}+h\Big[\frac{1}{2}(\frac{q_j-q_{j-1}}{h})^2-V(q_{j-1})-\frac{1}{2}\alpha s_{j-1}^2\Big]-\sin (q_{j-1}) \circ \Delta W_{j-1}.
\end{aligned}
\right.
\end{equation}
  To illustrate the performances of contact scheme $(\ref{4.7})$, the initial conditions are the same as those in Section 4.1.

 \begin{figure}[H]
   \centering
    \begin{minipage}{6.5cm}
       \includegraphics[width=2.8in, height=2.0in]{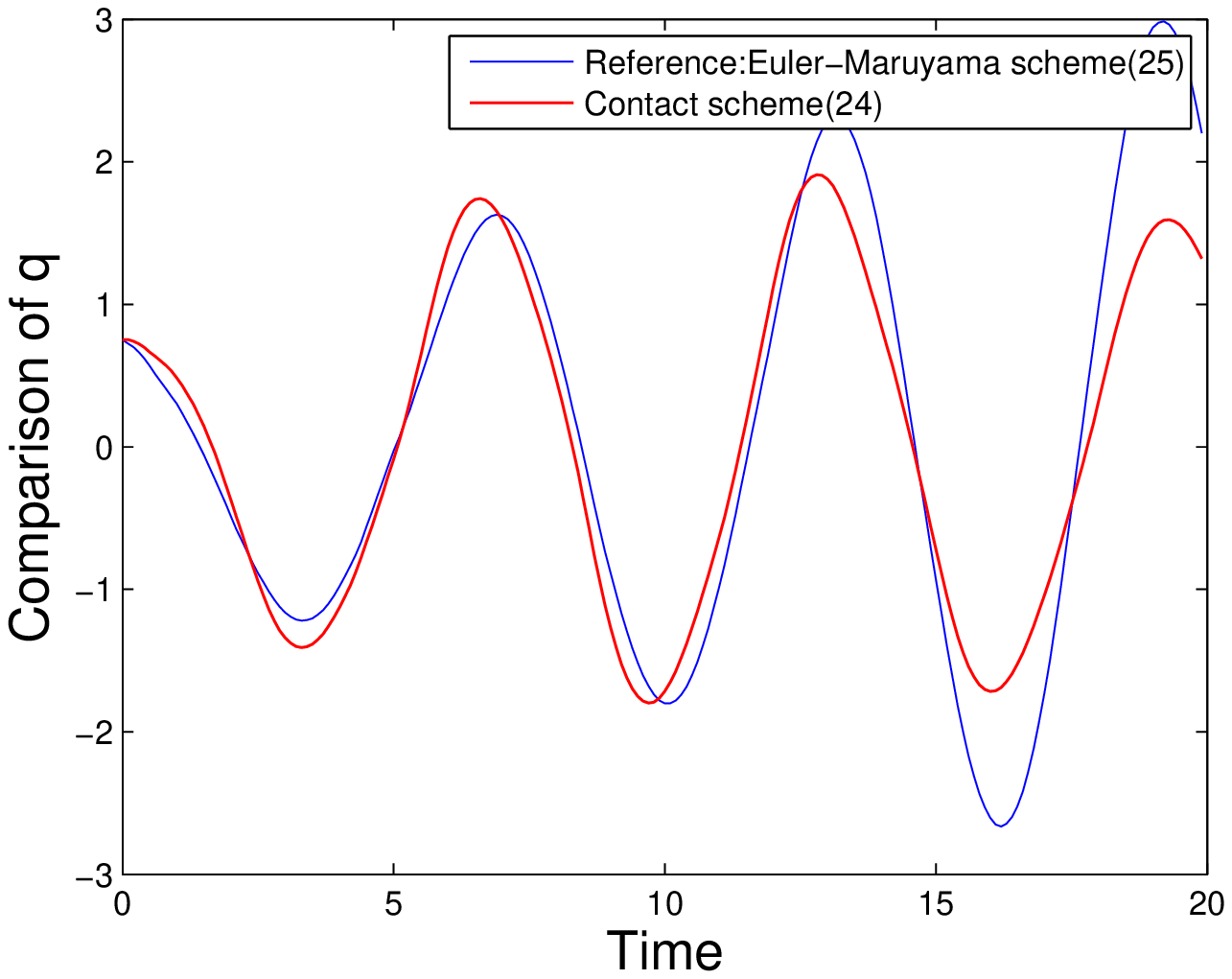}
    \end{minipage}
   \\
   \begin{minipage}{6.5cm}
       \includegraphics[width=2.8in, height=2.0in]{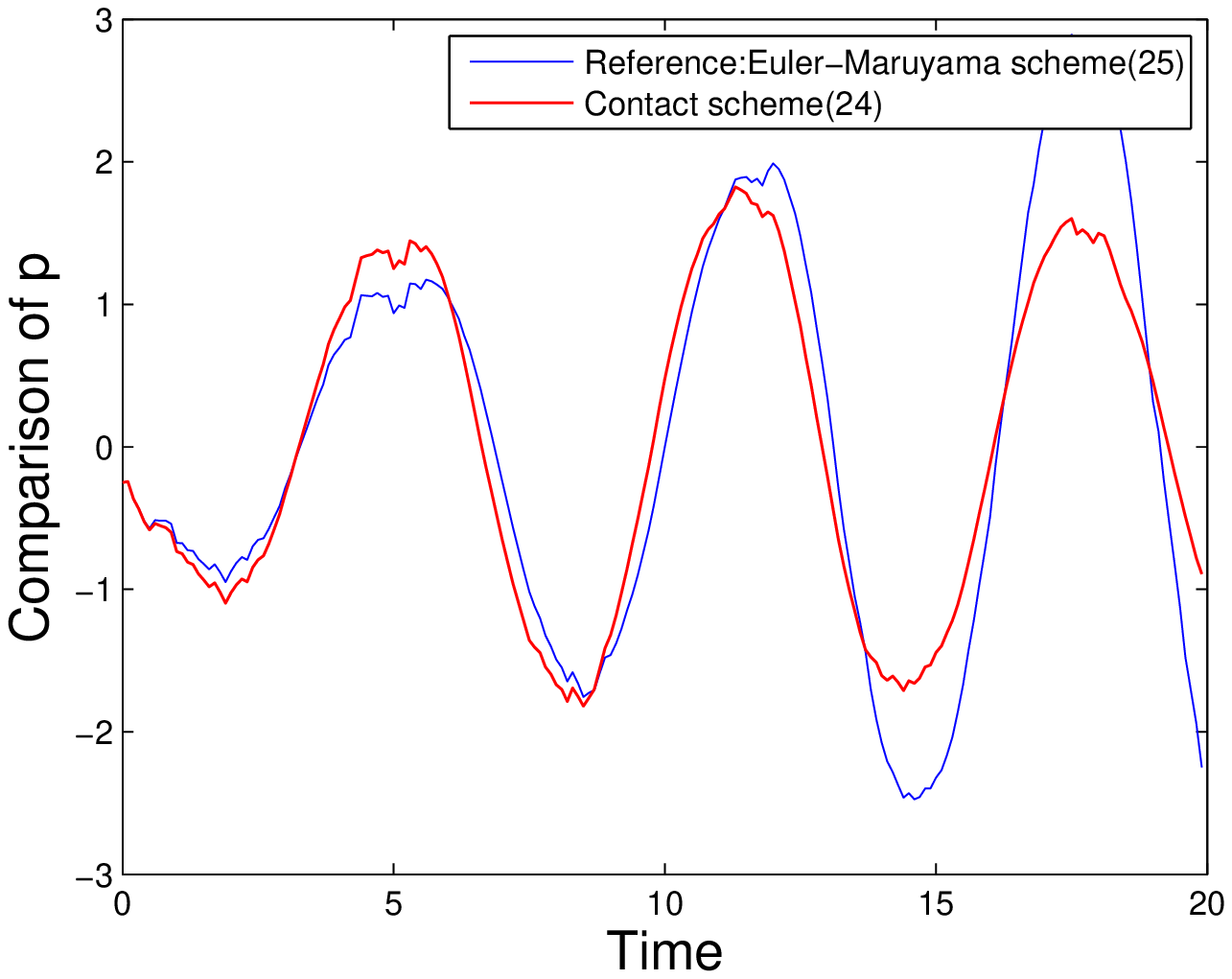}
    \end{minipage}
    \begin{minipage}{6.5cm}
       \includegraphics[width=2.8in, height=2.0in]{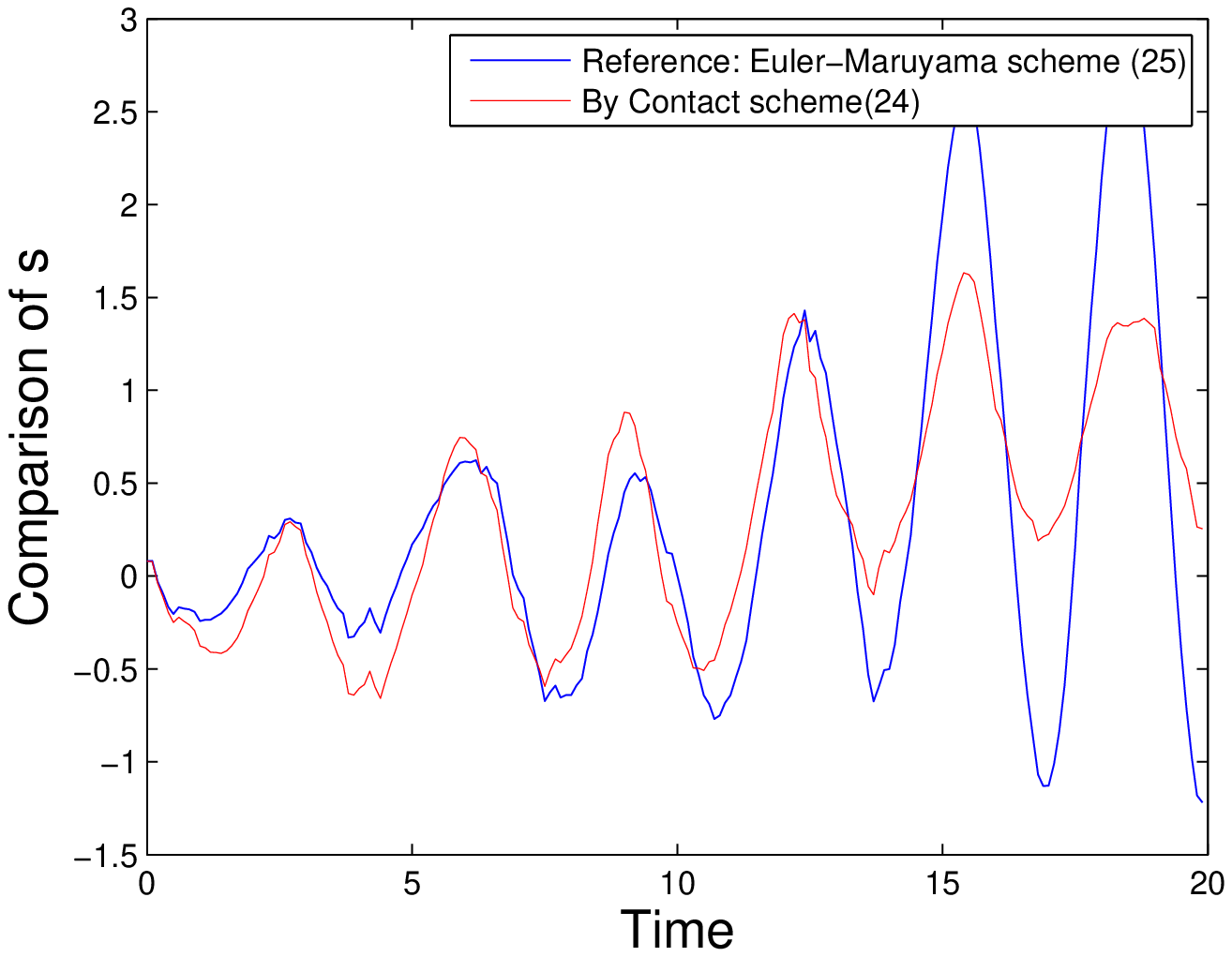}
    \end{minipage}
    \par{\scriptsize  Fig.4. Comparison of sample trajectories of $(\ref{4.7})$  and $(\ref{4.8})$ in the coordinates $q,p$ and $s$, respectively.}
\end{figure}
As we can see from Fig.4, the contact scheme has better performance than the non-contact scheme in the simulation of dissipation.

\begin{figure}[H]
\centering
\includegraphics[width=3.8in, height=1.8in]{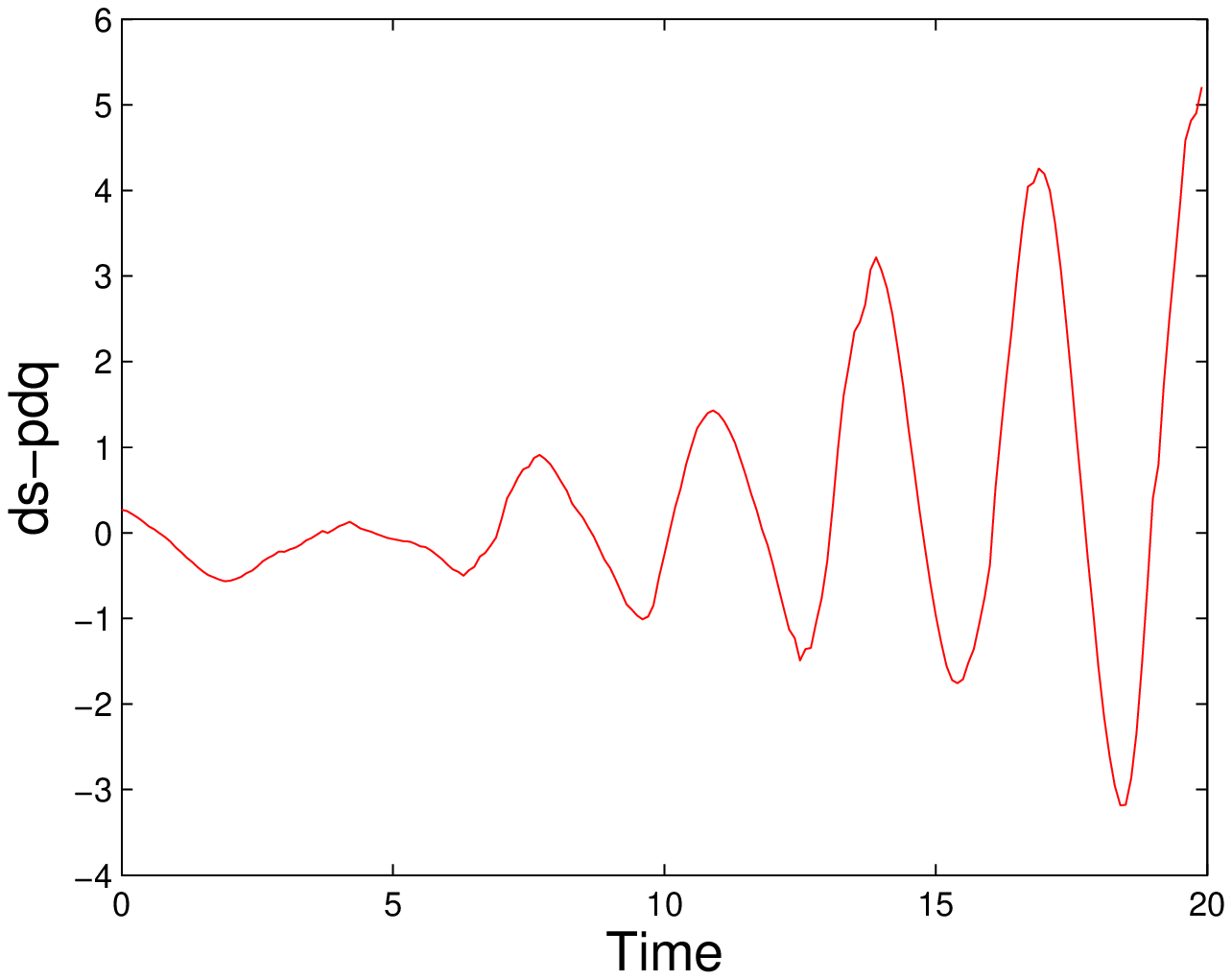}
\par{\scriptsize   Fig.5.  Preservation of contact structure $(\ref{4.5})$ .}
\end{figure}

 \begin{figure}[H]
\centering
\includegraphics[width=3.8in, height=2.2in]{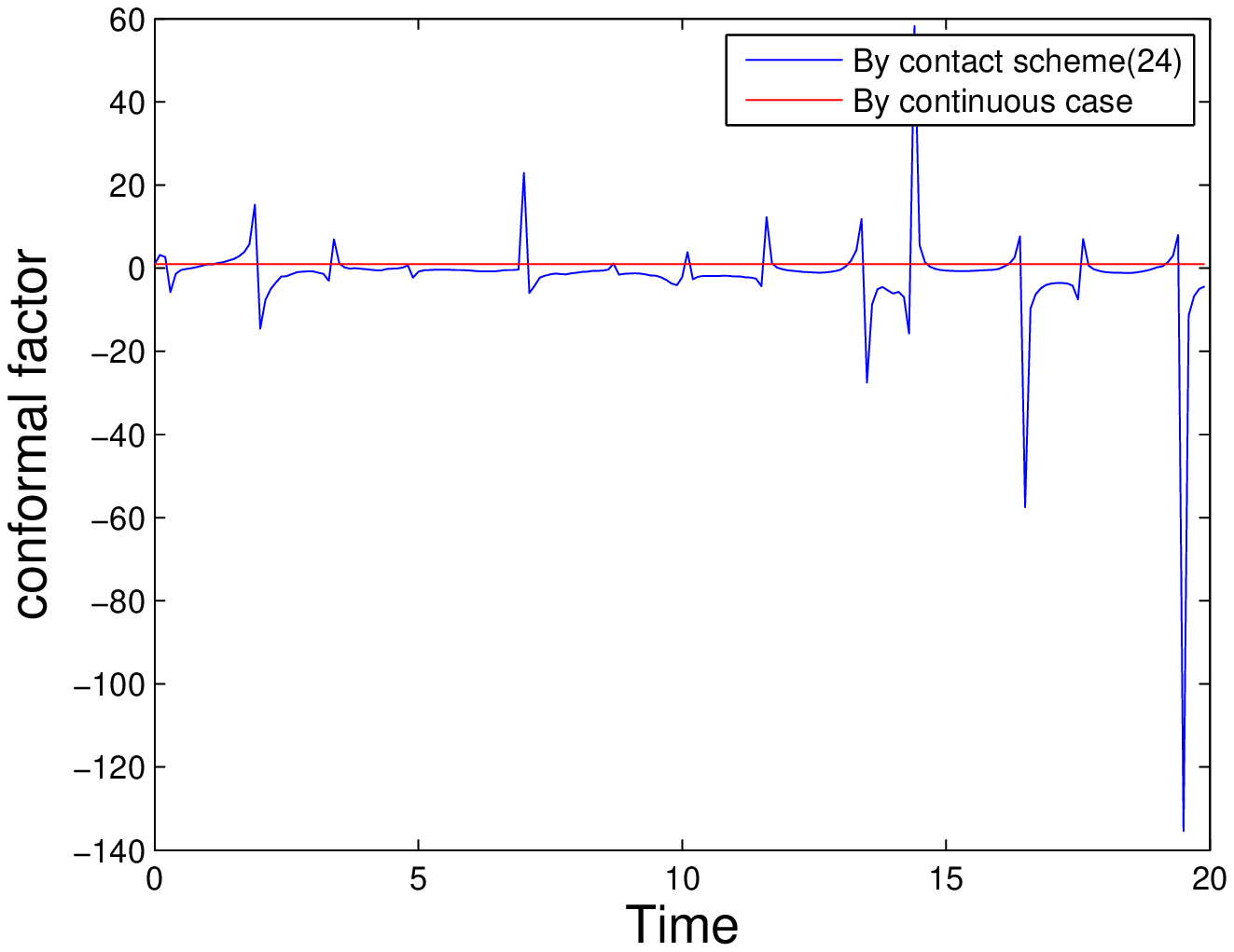}
\par{\scriptsize   Fig.6.  Comparison of the conformal factor of contact scheme $(\ref{4.7})$ and the continuous case of $(\ref{4.5})$.}
\end{figure}

It follows from Fig.5 and Fig.6 that the contact structure and conformal factor are preserved well almost surely, too. Therefore, these phenomena verify the results of Theorem 3.8. That is, it is the preservation of the contact structure that makes the outperformance.

\subsection{Example 3}
Now we consider the following stochastic Kepler problem equations \cite{Margheri},
\begin{equation}\label{4.9}
 \left \{
\begin{aligned}
dq=&pdt,\\
dp=&-\Big[\frac{1}{q^2}+\beta p\Big ]dt-\gamma \circ dW(t),\\
ds=&\Big[\frac{p^2}{2}+\frac{1}{|q|}-\beta s\Big ]dt-\gamma q \circ dW(t),
 \end{aligned}
\right.
 \end{equation}
where  $q,p,s\in \mathbb{R}$. The Hamiltonians are
 $$H_0=\frac{|p|^2}{2}-\frac{1}{|q|}+\beta s,\ \ \ H_1= \gamma q.$$
The Lagrangian function $L$ with respect to the deterministic part of $(\ref{4.9})$ is
$$L=p\dot{q}-H_0=\frac{\dot{q}^2}{2}+\frac{1}{|q|}-\beta s.$$

We consider one discrete form of the Lagrangian
\begin{equation}\label{4.10}
\begin{split}
L_j=\frac{1}{2}(\frac{q_{j+1}-q_j}{h})^2+\frac{2}{|q_j+q_{j+1}|}-\frac{1}{2}\beta(s_j+s_{j+1}).
 \end{split}
\end{equation}
Substituting the expression of $(\ref{4.10})$ into the relations $(\ref{3.5})$ and $(\ref{3.6})$, we have
\begin{equation}
\begin{aligned}
p_j=\frac{\frac{q_j-q_{j-1}}{h}-\frac{2h}{(q_j+q_{j-1})^2}} {1+\frac{1}{2}\beta h}.
\end{aligned}
\nonumber
\end{equation}
and
\begin{equation}
\begin{aligned}
p_{j-1}
=\frac{\frac{q_j-q_{j-1}}{h}+\frac{2h}{(q_{j-1}+q_j)^2}+\gamma \circ\Delta W_{j-1}}{1-\frac{1}{2}\beta h}.
\end{aligned}
\nonumber
\end{equation}

Therefore, an implicit contact scheme of $(\ref{4.9})$ is
\begin{equation}\label{4.11}
\left \{
\begin{aligned}
q_j=&q_{j-1}+\frac{h}{2}\Big[(1+\frac{1}{2}\beta h)p_j+(1-\frac{1}{2}\beta h)p_{j-1}-\gamma \circ\Delta W_{j-1}\Big ],\\
p_j=&\Big[\frac{q_j-q_{j-1}}{h}-\frac{2h}{(q_j+q_{j-1})^2}\Big]/(1+\frac{1}{2}\beta h),\\
s_j=&s_{j-1}+\frac{1}{2h}(q_j-q_{j-1})^2-\frac{2h}{|q_j+q_{j-1}|}-\frac{\beta h}{2}(s_j+s_{j-1})\\
   &-\gamma q_{j-1}\circ\Delta W_{j-1}.
\end{aligned}
\right.
\end{equation}

 Here we prove the preservation of contact structure of the scheme $(\ref{4.11})$, whose conformal factor is $\lambda_j=\exp(-\beta \Delta t_j)$ in the continuous case, and the comparison with Euler-Maruyama scheme. The Euler-Maruyama scheme of $(\ref{4.9})$ is

\begin{equation}\label{4.12}
\left \{
\begin{aligned}
q_j&=q_{j-1}+hp_{j-1},\\
p_j&=p_{j-1}-h[\frac{1}{q_{j-1}^2}+\beta p_{j-1}]-\gamma \circ  \Delta W_{j-1},\\
s_j&=s_{j-1}+h[\frac{1}{2}p_{j-1}^2+\frac{1}{|q_{j-1}|}-\beta s_{j-1}]-\gamma q_{j-1} \circ \Delta W_{j-1}.
\end{aligned}
\right.
\end{equation}
  To examine the performances of contact scheme $(\ref{4.11})$, the initial conditions are
$(q_0,p_0,s_0)=(0.75,-0.25,0.08),\beta=0.01,\gamma=0.1,h=0.1,N=2000,$
 \begin{figure}[H]
   \centering
    \begin{minipage}{6.5cm}
       \includegraphics[width=2.8in, height=2.0in]{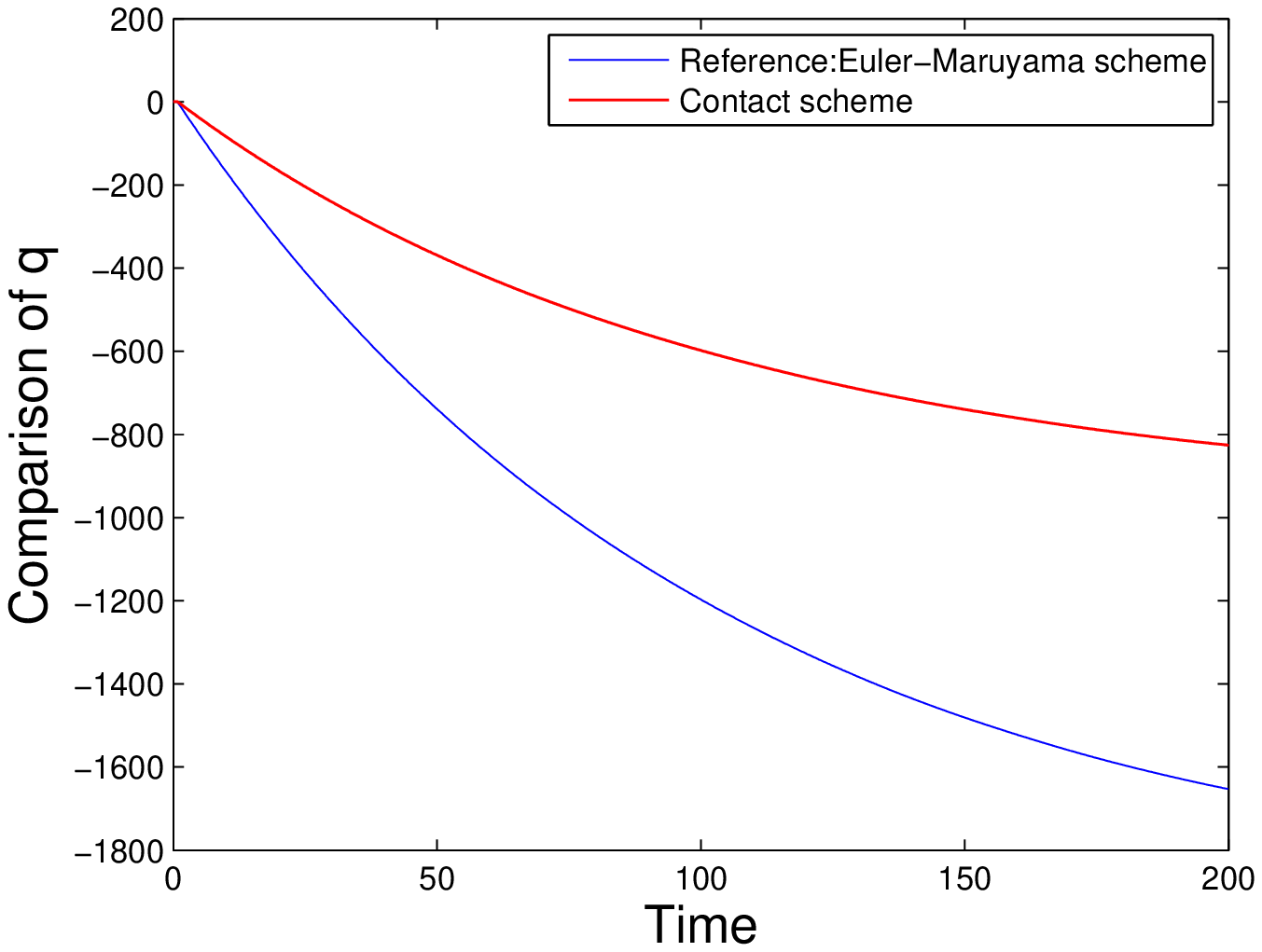}
    \end{minipage}
   \\
   \begin{minipage}{6.5cm}
       \includegraphics[width=2.8in, height=2.0in]{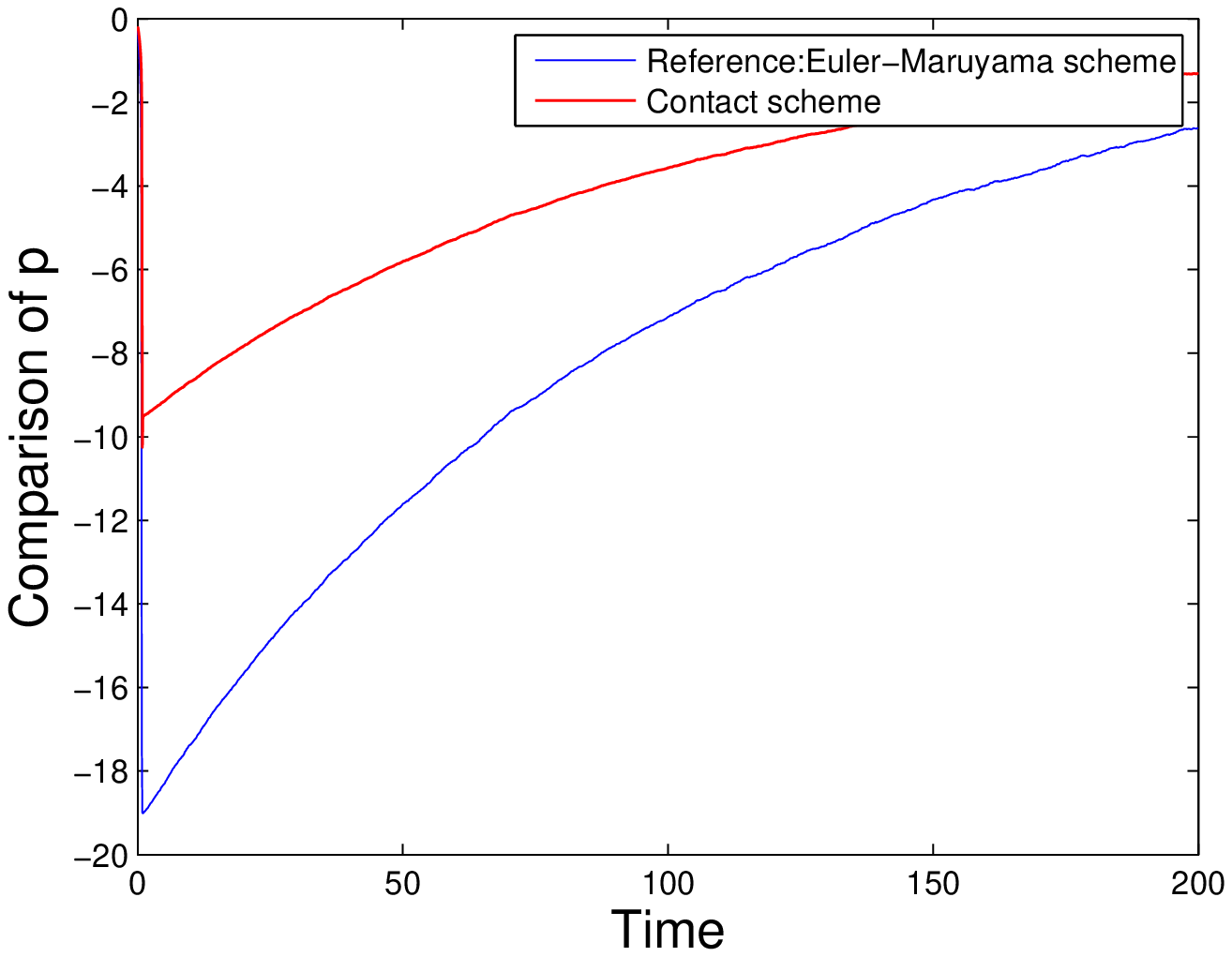}
    \end{minipage}
    \begin{minipage}{6.5cm}
       \includegraphics[width=2.8in, height=2.0in]{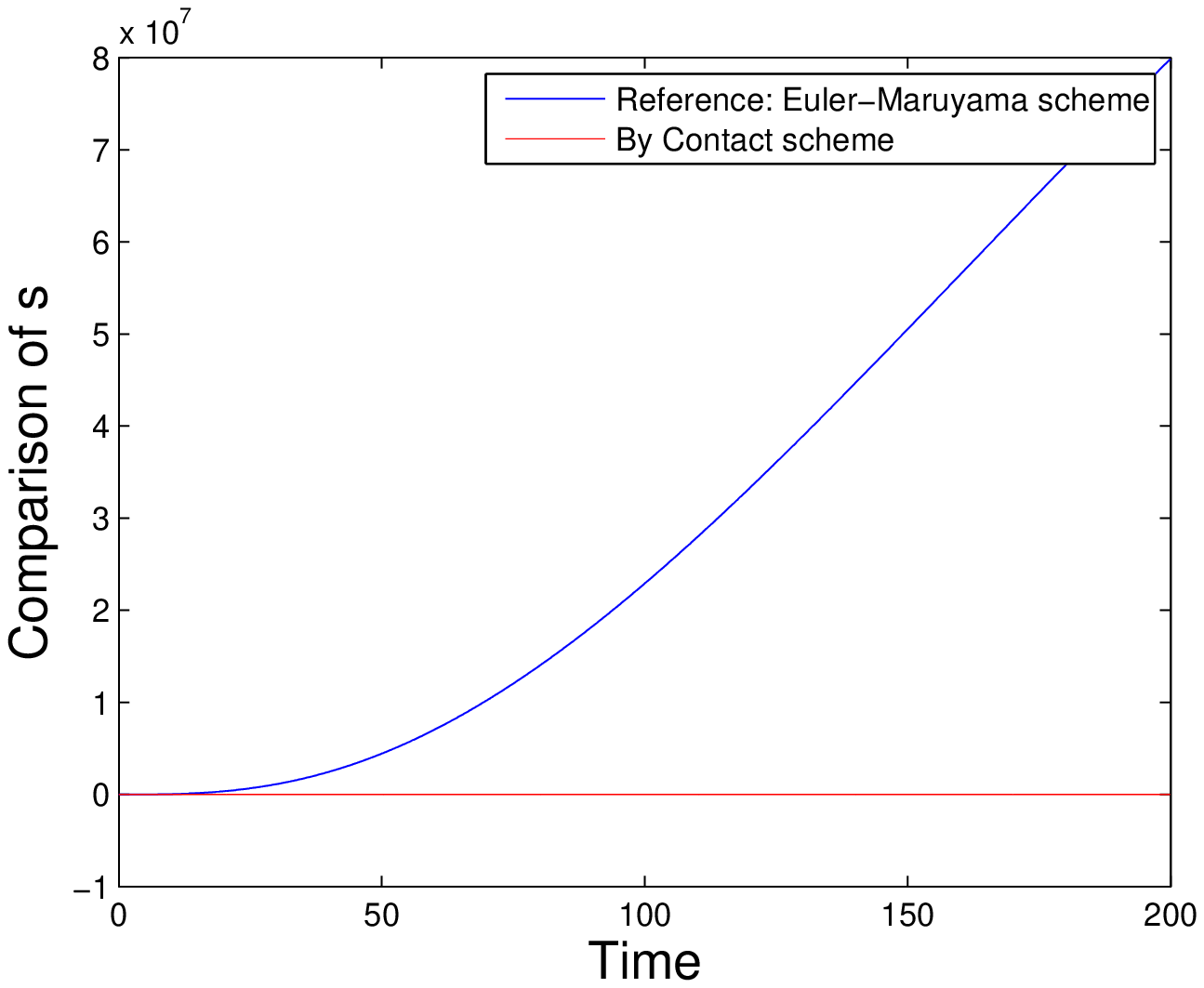}
    \end{minipage}
    \par{\scriptsize  Fig.7. Comparison of sample trajectories of $(\ref{4.11})$  and $(\ref{4.12})$ in the coordinates $q,p$ and $s$, respectively.}
\end{figure}
As we can see from Fig.7, the contact scheme has better performance than the non-contact scheme in the simulation of dissipation.

\begin{figure}[H]
\centering
\includegraphics[width=3.8in, height=1.8in]{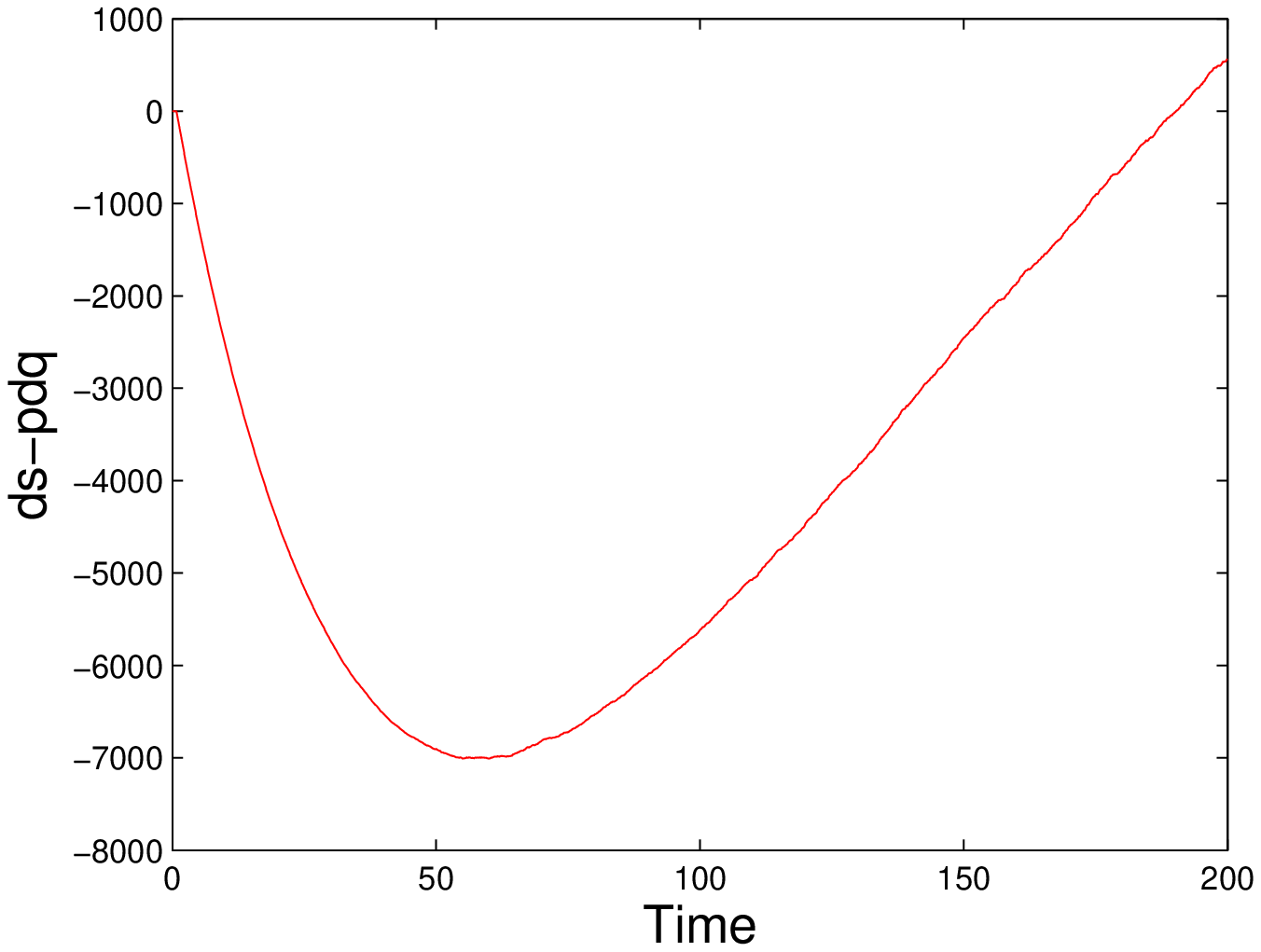}
\par{\scriptsize   Fig.8.  Preservation of contact structure $(\ref{4.9})$ .}
\end{figure}

 \begin{figure}[H]
\centering
\includegraphics[width=3.8in, height=2.2in]{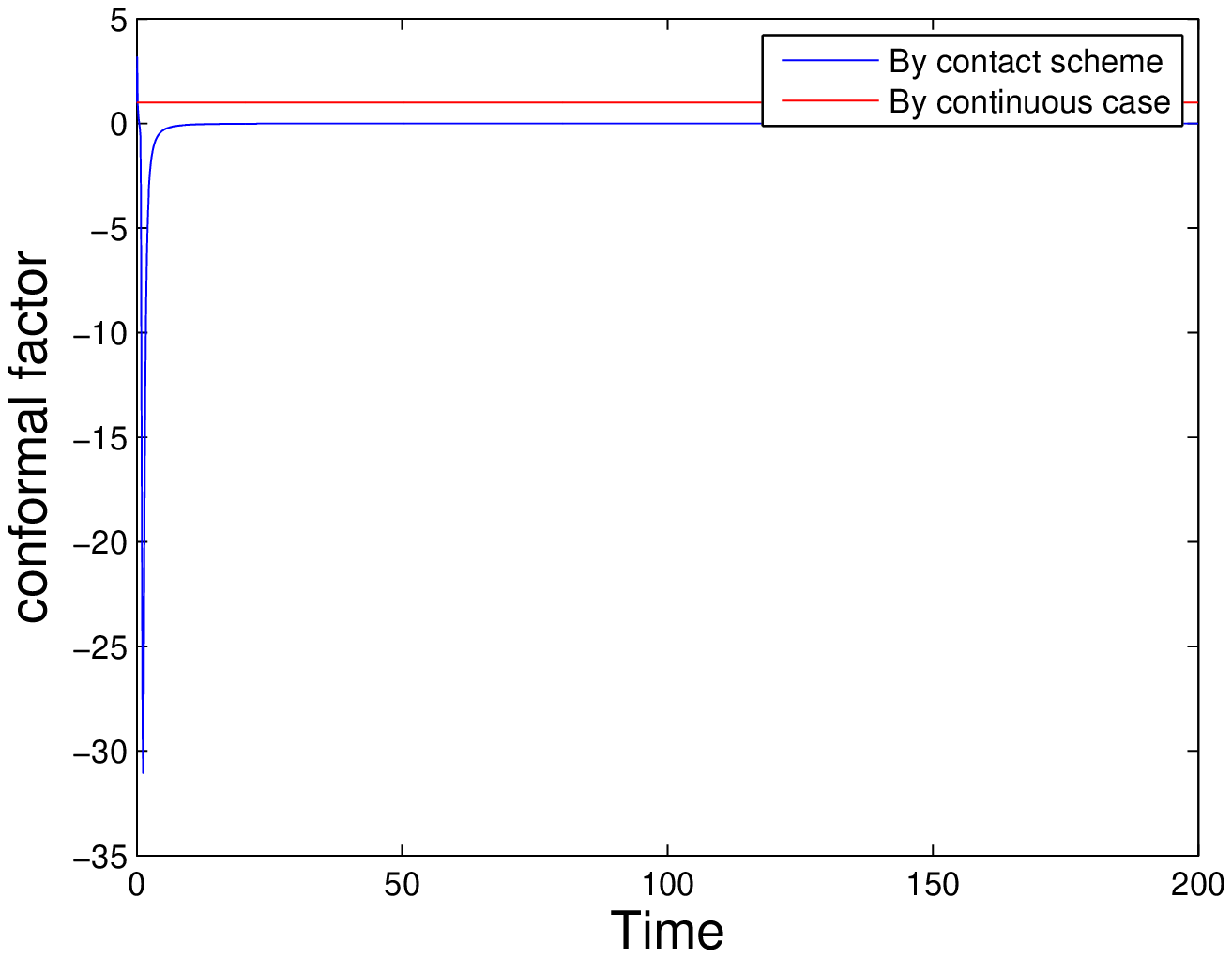}
\par{\scriptsize   Fig.9.  Comparison of the conformal factor of contact scheme $(\ref{4.11})$ and the continuous case of $(\ref{4.9})$.}
\end{figure}

It follows from Fig.8 and Fig.9 that the contact structure and conformal factor are preserved well almost surely, too. Therefore, these results verify the results of Theorem 3.8.

\begin{remark}

The construction of the Lagrangian of stochastic contact Hamiltonian systems is still an open problem. Here we only present some examples whose Lagrangian can be explicitly expressed via the deterministic part of the systems.
\end{remark}

\section{Conclusion}
 This paper focuses on the construction and proof of the stochastic contact variational integrator via stochastic Herglotz variational principle. We investigate the dynamics of stochastic contact Hamiltonian systems and the validation through performing three numerical experiments. These numerical experiments are conducted to demonstrate the effectiveness and superiority of the proposed method by the simulations of its orbits, contact structure and conformal factor over a long time interval. The results show that the method is effective as the numerical experiments match the results of theoretical analysis. Construction of various methods via the stochastic contact variational integrator theory would be our further work.

\section*{Statements}
All data in this manuscript is available. And all programs will be available on the WEB GitHub\cite{Zhan2}.

\end{document}